\documentclass[11pt]{amsart}

\usepackage{epigamath}

\usepackage[english]{babel}

\numberwithin{equation}{section}

\usepackage{enumitem}

\usepackage[T1]{fontenc}
\usepackage{url}
\usepackage{amssymb}
\usepackage{color}
\usepackage{tkz-fct}
\usepackage{tikz}
\usepackage{tikz-cd}
\usetikzlibrary{matrix}
\usetikzlibrary{arrows}
\usepackage{graphicx}
\usepackage{cleveref}
\usepackage{mathabx}

\usepackage[cal=boondoxo]{mathalfa}

\theoremstyle{plain}
\newtheorem{theoremintro}{Theorem}

\newtheorem{theorem}{Theorem}[section]
\newtheorem{lemma}[theorem]{Lemma}
\newtheorem{proposition}[theorem]{Proposition}
\newtheorem{corollary}[theorem]{Corollary}

\theoremstyle{definition}
\newtheorem{definition}[theorem]{Definition}

\theoremstyle{remark}
\newtheorem{remark}[theorem]{Remark}
\newtheorem{example}[theorem]{Example}

\setlist[enumerate,1]{label={\rm(\arabic*)}, ref={\rm\arabic*}}

\tikzcdset{arrow style=math font}

\newcommand{\bbA}{\mathbb{A}}

\newcommand{\bbC}{\mathbb{C}}

\newcommand{\bbG}{\mathbb{G}}

\newcommand{\bbP}{\mathbb{P}}

\newcommand{\bbV}{\mathbb{V}}

\newcommand{\bbZ}{\mathbb{Z}}

\newcommand{\Gm}{\mathbb{G}_m}

\newcommand{\cB}{\mathcal{B}}

\newcommand{\cF}{\mathcal{F}}

\newcommand{\cJ}{\mathcal{J}}

\newcommand{\cL}{\mathcal{L}}

\newcommand{\cO}{\mathcal{O}}

\newcommand{\cS}{\mathcal{S}}

\newcommand{\cU}{\mathcal{U}}

\newcommand{\cX}{\mathcal{X}}

\newcommand{\rB}{\textup{B}}

\newcommand{\rE}{\textup{E}}

\newcommand{\rG}{\textup{G}}
\newcommand{\rH}{\textup{H}}

\newcommand{\rR}{\textup{R}}

\newcommand{\rd}{\textup{d}}

\newcommand{\an}{\textup{an}}

\newcommand{\too}{\longrightarrow}
\renewcommand{\phi}{\varphi}
\renewcommand{\epsilon}{\varepsilon}
\renewcommand{\ker}{\Ker}

\DeclareMathOperator{\rk}{rk}
\DeclareMathOperator{\pr}{pr}

\DeclareMathOperator{\cHom}{\mathcal{H}\hspace{-2.5pt}\textit{om}}
\DeclareMathOperator{\cEnd}{\mathcal{E}\hspace{-1pt}\textit{nd}}
\DeclareMathOperator{\Hom}{Hom}

\DeclareMathOperator{\im}{Im}
\DeclareMathOperator{\Ker}{Ker}

\DeclareMathOperator{\coker}{Coker}

\DeclareMathOperator{\Sk}{Sk}

\DeclareMathOperator{\Pic}{Pic}

\DeclareMathOperator{\End}{End}
\DeclareMathOperator{\Ext}{Ext}

\DeclareMathOperator{\Lie}{Lie}

\DeclareMathOperator{\ev}{ev}
\DeclareMathOperator{\id}{id}

\DeclareMathOperator{\At}{At}

\DeclareMathOperator{\tr}{tr}

\renewcommand{\le}{\leqslant}
\renewcommand{\ge}{\geqslant}

\def\lowsim{\vbox to 0pt{\vss\hbox{$\scriptstyle\sim$}\vskip-1.5pt}}

\newcommand{\supth}[1]{\ensuremath{#1^{\mathrm{th}}}}
 
\EpigaVolumeYear{8}{2024} \EpigaArticleNr{4} \ReceivedOn{December 13, 2022}
\InFinalFormOn{October 20, 2023}
\AcceptedOn{November 21, 2023}

\title{The universal vector extension of an abeloid variety}
\titlemark{The universal vector extension of an abeloid variety}

\author{Marco Maculan}
\address{Institut de Math\'ematiques de Jussieu, Sorbonne Universit\'e, 4 place Jussieu, 75252 Paris, France}
\email{marco.maculan@imj-prg.fr}

\authormark{M.~Maculan}

\AbstractInEnglish{Let $A$ be an abelian variety over a complete non-Archimedean field $K$. The universal cover of the Berkovich space attached to $A$ reflects the reduction behaviour of $A$. In this paper the universal cover of the universal vector extension $E(A)$ of $A$ is described. In a forthcoming paper this will be one of the crucial tools to show that rigid analytic functions on $E(A)$ are all constant.}

\MSCclass{14G22, 14K05}
\KeyWords{Rigid analytic uniformization, universal vector extension, abeloid variety}

\acknowledgement{This research was supported by ANR-18-CE40-0017.}

\begin{document}

%%%%%%%%%%%%%%%%%%%%%%%%%%%%%%%
% Title page
%%%%%%%%%%%%%%%%%%%%%%%%%%%%%%%

\maketitle

\begin{prelims}

\DisplayAbstractInEnglish

\bigskip

\DisplayKeyWords

\medskip

\DisplayMSCclass

\end{prelims}

%%%%%%%%%%%%%%%%%%%%%
% Table of Contents
%%%%%%%%%%%%%%%%%%%%%

\newpage

\setcounter{tocdepth}{1}

\tableofcontents

%%%%%%%%%%%%%%%%%%%%%
% Content begins here
%%%%%%%%%%%%%%%%%%%%%

\section{Introduction}

\subsection{Background} Let $K$ be an algebraically closed non-trivially valued complete non-Archimedean field. The driving force behind Tate's foundation \cite{TateRigid} of rigid analysis was the uniformization of elliptic curves over a $p$-adic field: given an elliptic curve $E$ over $K$, there is an isomorphism $E(K) \cong K^{\times} / q^{\bbZ}$ of rigid analytic spaces for some $|q| < 1$ if and only if the $j$-invariant of $E$ is not integral.  Within the framework of Berkovich spaces, analytic spaces are locally path connected and locally contractible (in contrast with the total disconnectedness of~$K$), so that the usual theory of universal covers and fundamental groups can be applied. Tate's result then can be restated as follows. The topological space underlying the Berkovich space $E^\an$ attached to $E$ is contractible if and only if $E$ has good reduction; if this is not the case, the universal cover of $E^\an$ is $\Gm^{\an}$, its fundamental group is identified with $q^{\bbZ} \subset K^\times$ for some $|q| < 1$, and $E^\an = \Gm^{\an} / q^\bbZ$. Mumford generalized Tate's theorem both for higher-genus curves, see \cite{MumfordCurves}, and higher-dimensional abelian varieties, see \cite{MumfordDegeneratingAbelianVarieties}. Given an abelian variety $A$ over $K$, the universal cover of $A^\an$ is of the form $E^\an$ for a semi-abelian variety
\begin{equation} \label{Eq:UniversalCoverIntro}
  0 \too T \too E \too B \too 0,
\end{equation}
where $T$ is a $K$-torus and $B$ an abelian variety with good reduction. The topological space $E^\an$ is again seen to be contractible, and the fundamental group of $A^\an$ is seen to be a free abelian group $\Lambda \subset E(K) $ of rank $\dim T$. Later on L\"{u}tkebohmert \cite[Corollary 7.6.5]{LutkebohmertCurves} obtained such a uniformization result for all \emph{abeloid varieties}, that is, proper smooth connected analytic groups over $K$~-- the rigid analytic analogue of a complex torus.

\subsection{Motivation}
In this paper the universal cover of the universal vector extension of an abelian variety is made explicit. %In a forthcoming paper this will be used to show that rigid analytic functions on the universal vector extension are all constant---this is in opposition to the complex picture, where the universal vector extension is a Stein space thus carries plenty of non-constant holomorphic functions. 
Recall that a vector  extension of an abelian variety $A$ over a field $k$ is a short exact sequence of algebraic  groups
\[ 0 \too \bbV(F) \too G \too A \too 0,\]
where $F$ is a finite-dimensional $k$-vector space and $\bbV(F)$ the vector group attached to it. Such an extension corresponds (up to isomorphism) to elements of the cohomology group $\rH^1(A, \cO_A) \otimes_K F$. Taking $F = \rH^1(A, \cO_A)^\vee$, the vector extension with isomorphism class $\id \in \End \rH^1(A, \cO_A)$,
\[  0 \too \bbV\left(\rH^1\left(A, \cO_A\right)^\vee \right) \too E(A) \too A \too 0, \]
is said to be the \emph{universal} one; it has the property that any other vector extension $G$ is obtained as the push-out of $E(A)$ along the linear map $\rH^1(A, \cO_A)^\vee \to F$ given by the isomorphism class of $G$. 

When $k = \bbC$, the uniformization of $E(A)$ admits a quite explicit description, which should hopefully clarify the statements in the rigid context. The exponential map $V:= \Lie A \to A(\bbC)$ is a universal cover of $A(\bbC)$. We identify the fundamental group $\pi_1(A(\bbC), 0)$ with the kernel of the exponential map $\Lambda \subseteq \Lie A$. Hodge theory permits one to see $\rH^1(A, \cO_A)^\vee$ as the complex vector space $\bar{V}$ conjugate to $V$; consequently, we have an inclusion $\theta_\Lambda \colon \Lambda \to \bar{V}$. Then,
\[ E(A)(\bbC) = (V \times \bar{V}) / \Lambda\]
 with $\Lambda$ embedded diagonally. The choice of a basis of $\Lambda$ induces a biholomorphism $E(A)(\bbC)  \cong (\bbC^\times)^{2g}$. In particular, the complex manifold $E(A)(\bbC)$  is Stein, meaning that it can be holomorphically embedded in $\bbC^n$ as a closed subspace for some $n \ge 0$. Such an embedding cannot be algebraic, as all algebraic functions on $E(A)$ are constant. In other words, $E(A)$ is an example of a non-affine algebraic variety whose associated complex space is Stein. 
 
 The motivation of the present paper comes from the study of the analogous question over $K$; see \cite{UniversalExtension}.

\subsection{Results} From now on let us work over $K$. To ease notation, the superscript `an' is dropped, and all algebraic varieties are treated as analytic spaces. Let
\[ 0 \too T \too E \too B \too 0\]
be the semi-abelian variety which is the universal cover of $A$ and $\Lambda \subseteq E(K)$ the fundamental group. In order to state the main result, it is necessary to also have at hand  the uniformization of the dual abelian variety $\check{A}$. See $B$ as the dual of its dual~$\check{B}$; then the group morphism $\Lambda \to B(K)$ defines a semi-abelian variety
\[ 0 \too \check{T} \too \check{E} \stackrel{\check{p}}{\too} \check{B} \too 0,\]
where $\check{T}$ is the $K$-torus with group of characters $\Lambda$. Let $\check{\Lambda}$ be the group of characters of the $K$-torus $T$. The datum of $\Lambda \subseteq E(K)$ induces an embedding of $\check{\Lambda} \subseteq \check{E}(K)$, and
\[\check{A}^\an = \check{E}^\an / \check{\Lambda}.\]
For an algebraic group $G$, let $\omega_G = (\Lie G)^\vee$ denote the dual of its Lie algebra. Since the quotient map $\check{E}^\an \to \check{A}^\an$ is \'etale, the spaces of invariant $1$-forms $\omega_{\check{E}}$ and $\omega_{\check{A}}$ are identified. The above short exact sequence of algebraic groups gives the following one: 
\[ 0 \too \omega_{\check{B}} \stackrel{\rd \check{p}}{\too} \omega_{\check{A}} \too \omega_{\check{T}} \too 0,\]
where $\rd \check{p}$ is the pull-back of $1$-forms along $\check{p}$. The first result concerns structure.

\begin{theoremintro} \label{Th:StructureUniversalCoverUnivExt} The universal cover $\tilde{E}(A)$ of\, $E(A)$ is contractible,~is the pull-back to~$\tilde{A}$ of\, $E(A)$ and is the push-out of\, $E(B) \times_B E$ along $\rd \check{p}$; i.e.\ there is a commutative and~exact diagram 
\[
\begin{tikzcd}
0 \ar[r] & \bbV\left(\omega_{\check{B}}\right) \ar[d, "\rd \check{p}"] \ar[r] & E(B) \times_B E \ar[r] \ar[d] & E \ar[r] \ar[d, equal] & 0 \\
0 \ar[r] & \bbV\left(\omega_{\check{A}}\right) \ar[r] & \tilde{E}(A) \ar[r] & E \ar[r] & 0\rlap{.}
\end{tikzcd}
\]
\end{theoremintro}

In the statement the isomorphisms $\omega_{\check{A}} \cong \rH^1(A, \cO_A)^\vee$ and $\omega_{\check{B}} \cong \rH^1(B, \cO_B)^\vee$ are understood. The fundamental group $\pi_1(E(A), 0)$ of $E(A)$ with base-point $0$ can be identified with a subgroup of (the $K$-rational points of) its universal cover $\tilde{E}(A)$. The projection $\tilde{E}(A) \to E$ given by Theorem~\ref{Th:StructureUniversalCoverUnivExt} then induces a map
\[ \phi \colon \pi_1(E(A), 0) \too \Lambda.\]
In order to understand how $\pi_1(E(A), 0)$ sits inside $\tilde{E}(A)$, note that Theorem~\ref{Th:StructureUniversalCoverUnivExt} gives an isomorphism $\coker(E(B) \times_B E \to \tilde{E}(A)) \cong \bbV(\omega_{\check{T}})$. Let
\[ \pr_u \colon \tilde{E}(A) \too \bbV\left(\omega_{\check{T}}\right)\]
be the induced projection. The image of $\pi_1(E(A), 0)$ is described by means of the \emph{universal vector hull} of the group $\Lambda$, that is, the $K$-linear map $\theta_\Lambda \colon \Lambda \to \omega_{\check{T}}$ defined as follows. By definition, $\check{T}$ is the $K$-torus with group of characters $\Lambda$. Seeing $\chi \in \Lambda$ as a character $\chi \colon \check{T} \to \Gm$, set
\[ \theta_\Lambda(\chi) := \chi^\ast \tfrac{\rd z}{z} \in \omega_{\check{T}},\]
where $z$ is the coordinate function on $\Gm$.

\begin{theoremintro}\label{Th:UniversalCoverUnipotentQuotient}  The map $\phi \colon \pi_1(E(A), 0) \to \Lambda$ is an isomorphism, and
\[ \pr_u{} \circ \phi^{-1} = \theta_\Lambda.\]
\end{theoremintro}

Concretely, when the abelian variety $A$ has good reduction, the above results simply say that $E(A)$ is contractible, so that it coincides with its universal cover. In the extremely opposite situation, when $A = T / \Lambda$ has totally degenerate reduction, the result is more interesting and gives the following description of $E(A)$:
\[ E(A) = \left(T \times \bbV\left(\omega_{\check{T}}\right)\right) / \{ \left(\chi, \theta_{\Lambda}(\chi)\right) : \chi \in \Lambda\}. \]
In the framework of $1$-motives, Theorems~\ref{Th:StructureUniversalCoverUnivExt} and~\ref{Th:UniversalCoverUnipotentQuotient} say that $\tilde{E}(A)$ is the universal vector extension of the $1$-motive $M = [\Lambda \to E]$; see \cite[Section~1.4]{BVS} and \cite[Section~2.2]{Ber}.

\subsection{Content of the paper} In order to show  Theorems~\ref{Th:StructureUniversalCoverUnivExt} and~\ref{Th:UniversalCoverUnipotentQuotient}, the natural approach would be to compare $\tilde{E}(A)$ with the universal vector extension $M^\natural$ of the $1$-motive $M = [\Lambda \to E]$. However, $\tilde{E}(A)$ is an \emph{analytic} vector extension of $E$, and since $E$ is not proper, it is not clear \textit{a priori} why it should be an \emph{algebraic} one. Therefore, we cannot apply the universal property of $M^\natural$ directly. Instead, we will proceed by making an explicit construction of $E(A)$ giving the main results above as a byproduct. Unfortunately, the universal property of the universal vector extension does not say much about how $E(A)$ is constructed. It is instead more insightful to look at the moduli space $A^\natural$ of translation-invariant line bundles on $A$ endowed with a (necessarily integrable) connection~-- such a moduli space is canonically isomorphic to the universal vector extension. The construction of $A^\natural$ presented here (see \Cref{Sec:UniversalExtensionAbelianScheme}), although quite natural, seemingly does not appear in the literature; it is inspired from \cite{Biswas}, even though the Hodge-theoretical reasoning therein had to be circumvented. Such a definition for $A^\natural$ has several advantages. First, it translates right away to the rigid analytic framework for abeloid varieties. Second, its explicit nature permits one to perform the necessary computations (see notably \Cref{Prop:LinearizationCanonicalExtension}, \Cref{Prop:UniversalCoverUniversalVectorExtension} and \Cref{Thm:LinearizationAndUniversalVectorHull}). Third, it allows one to determine the canonical linear isomorphism through which $A^\natural$ is obtained by push-out from $E(A)$ (see \Cref{Thm:CanonicalIsomorphismUniversalVectorExtensionModuliSpaceConnections}).\footnote{This formula, perhaps because of its what-else-could-it-be nature, is missing in the literature. Mazur and Messing prove the canonical isomorphism is functorial in $A$ (see \cite[Proposition 2.6.7]{MazurMessing}), but its explicit form  is missing. This gap was already pointed out by Crew (see \cite[Introduction]{CrewCompositio}), but in the description he proposes (\emph{op.cit.}, Theorem 2.7), he does not determine such a linear map.} The `universal cover' of $A^\natural$ is then defined by hand (see \Cref{sec:UniversalCoverUniversalVectorExtension}) and only ultimately shown to be contractible (see \Cref{Prop:UniversalCoverContractible}), so that it is literally the universal cover of $A^\natural$.

\subsection{Conventions}Let $X$ be a locally ringed space and
\begin{equation} \tag{$F$}
\cdots \too F_{i - 1} \too F_i \too F_{i + 1} \too \cdots
\end{equation}
a sequence of $\cO_{X}$-modules indexed by integers. For morphisms of locally ringed spaces $f \colon Y \to X$ and $g \colon X \to Z$, and an $\cO_X$-module $M$, let $f^\ast (F)$, $g_\ast (F)$, $(F) \otimes M$ denote the sequences obtained from $(F)$ by, respectively, pulling back along $f$, pushing forward along $g$ and taking the tensor product with $M$.

Let $K$ be a complete non-trivially valued non-Archimedean  field. In this paper $K$-analytic spaces are considered in the sense of Berkovich (see \cite{BerkovichIHES}). 
By abuse of notation, given a $K$-analytic space $X$, an $\cO_X$-module here is what is called an $\cO_{X_{\rG}}$-module in \emph{op.cit.} As soon as the $K$-analytic space $X$ is good (that is, every point admits an affinoid neighbourhood), the  two notions coincide (\emph{op.cit.}, Proposition 1.3.4). For a $K$-analytic space $S$, an $S$-analytic space in groups will be called simply an $S$-analytic group. An \emph{abeloid variety} over $S$ is a proper, smooth $S$-analytic group with connected fibers.

\section{Moduli of rank 1 connections and universal extension of an abelian scheme} \label{Sec:UniversalExtensionAbelianScheme}

Let $S$ be a scheme, $\alpha \colon A \to S$ an abelian scheme, $e \colon S \to A$ the zero section and $\omega_{A} = e^\ast \Omega_{A/S}^1$. Let $\mu, \pr_1, \pr_2 \colon A \times_S A \to A$ be, respectively, the group law, the first  and the second projections.

\subsection{Connections on homogeneous line bundles} \label{Sec:ConnectionsHomogeneousLineBundles}
A \emph{homogeneous line bundle} on $A$ is the datum of a line bundle on $A$ together with an isomorphism of $\cO_A$-modules $\phi \colon  \pr_1^\ast L \otimes \pr_2^\ast L \to \mu^\ast L$. The isomorphism $\phi$ is called \emph{rigidification}.

\begin{remark} \label{Rmk:RigidificationAtNeutralSection} The isomorphism $(e, e)^\ast \phi \colon e^\ast L \otimes e^\ast L \to e^\ast L$ induces a trivialization $u \colon \cO_S \to e^\ast L$ of the line bundle $e^\ast L$ on $S$.
\end{remark}

Let $\Delta_1$ be the first-order thickening of $A \times_S A$ along the diagonal, $p_i \colon \Delta_1 \to A$ the $\supth{i}$ projection  for $i = 1, 2$, $A_1$ the first-order thickening of $A$ along $e$, $\iota \colon A_1 \to A$ the closed immersion, $\pi \colon A_1 \to S$ the structural morphism and $\tau \colon A_1 \to \Delta_1$ the morphism defined by $p_{1} \circ \tau = x \circ \pi$ and $p_{2} \circ \tau = \iota$.  Following Grothendieck, a connection on $L$ is an isomorphism of $\cO_{\Delta_1}$-modules $\nabla \colon p_1^\ast L \to p_2^\ast L$ whose restriction to $A$ is the identity. Similarly, an infinitesimal rigidification\footnote{Mazur--Messing call an `infinitesimal rigidification' simply a `rigidification'. Here, the adjective `infinitesimal' is added in order to distinguish the concept from that of a rigidification of a homogeneous line bundle.} at $e$ is an isomorphism of $\cO_{A_1}$-modules $\rho \colon \pi^\ast e^\ast \to \iota^\ast L$ whose restriction to the zero section is the identity. (Refer to Sections~\ref{sec:AtiyahExtensions} and~\ref{sec:InfinitesimalRigidifications}
for basics on connections and infinitesimal rigidifications.) In particular, for a connection $\nabla$ on $L$, $\tau^\ast \nabla$ is an infinitesimal rigidification of $L$ at $e$.

\begin{proposition} \label{Lemma:ConnectionsAndRigidifications}  For a homogeneous line bundle $L$ on $A$, we have a bijection
\[ \left\{ \textup{connections on $L$} \right\} \too
\left\{
\begin{array}{c}
\textup{infinitesimal} \\
\textup{rigidifications of $L$ at $e$} 
\end{array} \right\}, \quad \nabla \longmapsto \tau^\ast \nabla.
 \]
\end{proposition}

\begin{proof} The argument is borrowed from (2) and (3) in the proof of \cite[Proposition 3.2.3, pp.~39--40]{MazurMessing}.

\textit{Injectivity.}~ Let $\nabla, \nabla' \colon p_1^\ast L \to p_2^\ast L$ be connections on $L$. Since both isomorphisms $\nabla$ and $\nabla'$ are the identity when restricted to the diagonal, they differ by an homomorphism of $\cO_X$-modules $L \to \Omega^1_{\alpha} \otimes L$. The latter is equivalent to the datum of a global section $\omega$ of $\Omega^1_{\alpha}$ as $L$ is a line bundle. On the other hand, the evaluation at $e$ homomorphism $\epsilon  \colon \alpha_\ast \Omega^1_\alpha \to \omega_A$ is also an isomorphism. Therefore, the isomorphisms $\tau^\ast \nabla$ and $\tau^\ast \nabla'$ differ by $\epsilon(\omega)$, which is $0$ if and only if $\omega$ is.

\textit{Surjectivity.}~ Let $\rho \colon \pi^\ast e^\ast L \to \iota^\ast L$ be an infinitesimal rigidification of $L$. The morphism $p_2 - p_1 \colon \Delta_1 \to A$ factors  through the closed immersion $\iota \colon A_1 \to A$. Indeed, when restricted to the diagonal, the map $p_2 - p_1$ has constant value $e$. As $A_1$ (resp.\ $\Delta_1$) is the first-order thickening of $A$ along $e$ (resp.\ of $X \times_S X$ along the diagonal $\Delta$), the morphism $p_2 - p_1$ induces a morphism $\eta \colon \Delta_1 \to A_1$ between first-order thickenings such that $\iota \circ \eta = p_2 - p_1$. The pull-back of the rigidification $\phi$ along the morphisms of $S$-schemes 
$  (e \circ \pi \circ \eta, p_1), (p_2 - p_1, p_1) \colon \Delta_1 \to A \times A$
furnishes the following isomorphisms of line bundles over $\Delta_1$: 
\begin{align*} 
(e \circ \pi \circ \eta, p_1)^\ast \phi \colon \eta^\ast \pi^\ast e^\ast L \otimes p_1^\ast L &\too p_1^\ast  L, \\
(p_2 - p_1, p_1)^\ast \phi \colon  (p_2 - p_1)^\ast L \otimes p_1^\ast L &\too p_2^\ast L.
 \end{align*}
On the other hand, taking the tensor product of $\eta^\ast \rho$ with $p_1^\ast L$  gives rise the isomorphism of $\cO_{\Delta_1}$-modules
\[ \eta^\ast \rho \otimes \id \colon \eta^\ast \pi^\ast e^\ast L \otimes p_1^\ast L \too (p_2 - p_1)^\ast L \otimes p_1^\ast L.\]
Consider the unique homomorphism of $\cO_{\Delta_1}$-modules $\nabla \colon p_1^\ast L \to p_2^\ast L$ making the following diagram commutative:
\begin{equation} \label{eq:DiagramDefiningConnectionFromRigidifcation}
\begin{tikzcd}[column sep=35pt]
\eta^\ast \pi^\ast e^\ast L \otimes p_1^\ast L \ar[r, "{\eta^\ast \rho \otimes \id}"]  \ar[d, "{(e \circ \pi \circ \eta, p_1)^\ast \phi}", swap] & (p_2 - p_1)^\ast L \otimes p_1^\ast L \ar[d, "{(p_2 - p_1, p_1)^\ast \phi}"] \\
p_1^\ast  L \ar[r, "\nabla"] & p_2^\ast L\rlap{.} 
\end{tikzcd}
\end{equation}
To conclude, one has to show that the infinitesimal rigidification $\tau^\ast \nabla$ is $\rho$. Notice that the endomorphism $\eta \circ \tau$ of the $S$-scheme $A_1$ is the identity. Indeed, since $\iota$ is a closed immersion (thus a monomorphism of schemes), it suffices to show the equality $\iota \circ \eta \circ \tau = \iota$. But $ \iota \circ \eta \circ \tau = (p_2 - p_1) \circ \tau = \iota - e \circ \pi = \iota$ as $A_1$-valued points of $A$ because $e \circ \pi$ is the neutral element of the group $A(A_1)$. Now, pulling back the diagram \eqref{eq:DiagramDefiningConnectionFromRigidifcation} along $\tau$ gives the following commutative diagram of $\cO_{A_1}$-modules:
\[ 
\begin{tikzcd}[column sep=35pt]
\pi^\ast e^\ast L \ar[r, "\tau^\ast \nabla"]\ar[d, "{\id \otimes \pi^\ast u}", swap] & \iota^\ast L \ar[d, "\id \otimes \pi^\ast u"] \\
\pi^\ast e^\ast L \otimes \pi^\ast e^\ast L \ar[r, "{\rho \otimes \id}"]  & \iota^\ast L \otimes \pi^\ast e^\ast L\rlap{,}
\end{tikzcd}
\]
where $u \colon \cO_S \to e^\ast L$ is the trivialization introduced in \Cref{Rmk:RigidificationAtNeutralSection}. (Here we used the equalities  $p_{1} \circ \tau = x \circ \pi$ and $p_{2} \circ \tau = \iota$ holding by the definition of $\tau$.) The equality $\tau^\ast \nabla = \rho$ follows. This concludes the proof.
\end{proof}

For a homogeneous line bundle $L$ on $A$, let
\begin{equation}
\tag{$\At_{A/S}(L)$} 0 \too \Omega^1_{A/S} \too \At_{A/S}(L) \too \cO_A \too 0
\end{equation}
be its Atiyah extension (see \Cref{sec:AtiyahExtensions}). Recall that connections on $L$ correspond to splittings of the Atiyah extension of $L$.

\begin{proposition} \label{Prop:VanishingAtiyahClass}
  Let $(L, \phi)$ be a homogeneous line bundle on $A$.
\begin{enumerate}
\item\label{p:vac-1} If the cohomology group $\rH^1(S, \omega_A)$ vanishes, then the line bundle $L$ admits a connection.
\item\label{p:vac-2} The following sequence of\, $\cO_S$-modules is short exact:
\begin{equation}
\tag*{$\alpha_\ast (\At_{A/S}(L))$}
0 \too \alpha_\ast \Omega^1_{A/S} \too \alpha_\ast \At_{A/S}(L) \too \alpha_\ast \cO_A \too 0.
\end{equation}
\item\label{p:vac-3} The homomorphism $\alpha_\ast \alpha^\ast (\At_{A/S}(L)) \to (\At_{A/S}(L))$ of short exact sequences of\, $\cO_A$-modules obtained by adjunction is an isomorphism.
\end{enumerate} 
\end{proposition}

\begin{proof} \eqref{p:vac-1}~ According to \Cref{Lemma:ConnectionsAndRigidifications}, it suffices to show that the line bundle $L$ admits an infinitesimal rigidification at $e$. With the notation of \Cref{Prop:RigidificationsAndSplittings}, an infinitesimal rigidification corresponds to a splitting of the short exact sequence of $\cO_S$-modules
\begin{equation}\tag*{$\pi_\ast (\iota^\ast L)$}
0 \too \omega_A \otimes e^\ast L \too \pi_\ast \iota^\ast L \too e^\ast L \too 0.
\end{equation}
The latter is an extension of the line bundle $e^\ast L$ on $S$ by the vector bundle $\omega_A \otimes e^\ast L$; therefore, its isomorphism class lies in $\rH^1(S, \cHom(e^\ast L, \omega_A \otimes e^\ast L)) = \rH^1(S, \omega_A)$. By hypothesis, the cohomology group $\rH^1(S, \omega_A)$ vanishes; hence the short exact sequence $\pi_\ast (\iota^\ast L)$ splits, and the line bundle $L$ admits a connection.

\eqref{p:vac-2}~ The push-forward along $\alpha$ is a left-exact functor; therefore, it suffices to show that the natural map $p \colon \alpha_\ast \At_{A/S}(L) \to \alpha_\ast \cO_A$ is surjective. The statement is local on $S$; thus the scheme $S$ may be supposed to be affine. Under this assumption, the cohomology group $\rH^1(S, \omega_A)$ vanishes, and because of \eqref{p:vac-1}, the line bundle $L$ admits a connection. In other words, by \Cref{Prop:CharacterizationConnections},  the Atiyah extension $(\At_{A/S}(L))$ of $L$  admits a splitting $s \colon \cO_A \to \At_{A/S}(L)$.  The homomorphism of $\cO_S$-modules $\alpha_\ast s \colon \alpha_\ast \cO_A \to \alpha_\ast \At_{A/S}(L)$ is a section of $p$; that is, it satisfies $p \circ \alpha_\ast s = \id$. In particular, the homomorphism $p$ is surjective.

\eqref{p:vac-3}~ The homomorphism of short exact sequences of $\cO_A$-modules in question is the commutative diagram
\[
\begin{tikzcd}
0 \ar[r] & \alpha^\ast \alpha_\ast \Omega^{1}_{A/S} \ar[r] \ar[d] & \alpha^\ast \alpha_\ast \At_{A/S}(L) \ar[r] \ar[d] & \alpha^\ast \alpha_\ast \cO_A \ar[r] \ar[d] & 0 \\
0 \ar[r] & \Omega^{1}_{A/S} \ar[r] & \At_{A/S}(L) \ar[r] & \cO_A \ar[r] & 0\rlap{,} 
\end{tikzcd}
\]
where the three vertical arrows are given by adjunction. The leftmost and the rightmost vertical arrows are isomorphisms; thus the central vertical arrow must be so by the five lemma.
\end{proof}

\begin{remark} \label{Remark:AlternativeDescriptionPushForwardAtiyahExtension} The Atiyah extension $(\At_{A/S}(L))$ is obtained as the tensor product with $L^\vee$ of the short exact sequence 
\begin{equation}
\tag{$\cJ^1_{A/S}(L)$} 0 \too \Omega^{1}_{A/S} \otimes L \too \cJ^1_{A/S}(L) \too L \too 0,
\end{equation}
where $\cJ^1_{A/S}(L)$ is the $\cO_X$-module of first-order jets of $L$. \Cref{Prop:RigidificationsAndPrincipalParts} furnishes an isomorphism 
$\phi \colon e^\ast (\cJ^1_{A/S}(L)) \to \pi_\ast (\iota^\ast L)$ of short exact sequences of $\cO_S$-modules. Taking the tensor product with $e^\ast L^\vee$ induces an isomorphism
\[ e^\ast (\At_{A/S}(L)) \too \pi_\ast (\iota^\ast L) \otimes e^\ast L^\vee\]
of short exact sequences of $\cO_S$-modules. Also note  that, by \Cref{Rmk:RigidificationAtNeutralSection}, the line bundle $e^\ast L$ on $S$ is trivial, whence we have an isomorphism
\[ e^\ast (\At_{A/S}(L)) \cong \pi_\ast (\iota^\ast L).\]
\end{remark}

\subsection{The canonical extension of the trivial line bundle}  The functor associating with an $S$-scheme $S'$ the set of isomorphism classes of homogeneous line bundles on $A \times_S S'$ is representable by an abelian scheme $\check{\alpha} \colon \check{A} \to S$, called the \emph{dual abelian scheme} (this is \cite[Theorem 1.9]{FaltingsChai}; to see that the definition in \emph{loc.~cit.} is equivalent to the one here, see \cite[Proposition 18.4]{OortCommutativeGroupSchemes}). Let $\cL$ be the Poincar\'e bundle (that is, the universal homogeneous line bundle on $A \times_S\check{A}$), $q \colon A \times_S \check{A} \to \check{A}$ the projection onto the second factor and $(\At_q(\cL))$ the Atiyah extension relative to $q$ of the Poincar\'e bundle $\cL$. Set $\cU_{\check{A}} := q_\ast \At_q(\cL)$. The sequence of $\cO_{\check{A}}$-modules 
\begin{equation} \tag{$\cU_{\check{A}}$}
  0 \too \check{\alpha}^\ast \omega_{A} \too \cU_{\check{A}} \too \cO_{\check{A}} \too 0
\end{equation}
 obtained by pushing forward the short exact sequence $(\At_q(\cL))$ along $q$, is short exact by \Cref{Prop:VanishingAtiyahClass}\eqref{p:vac-2} applied  to the abelian scheme $A \times_S \check{A}$ over $\check{A}$.

\begin{definition} \label{Def:UniversalExtension} The extension $(\cU_{\check{A}})$ is called the \emph{canonical extension} of $\cO_{\check{A}}$.
\end{definition}

\begin{remark} \label{Rmk:SplittingCanonicalExtensionByCanonicalDerivation} The pull-back of the Atiyah extension of $\cL$ along the unramified morphism $(\id_{A}, \check{e})$ is the Atiyah extension of the trivial line bundle $\cO_{A}$. Therefore, the canonical derivation $\rd_{A/S} \colon \cO_{A} \to \Omega^1_{A/S}$, being a connection on the trivial bundle, defines a splitting of $\check{e}^\ast (\cU_{\check{A}})$, called the \emph{canonical splitting}.
\end{remark}

\begin{remark} \label{Rmk:AlternativeDescriptionCanonicalExtension} Let $A_1$ be the first-order thickening of $A$ along the zero section $e$, $\iota \colon A_1 \times_S \check{A} \to A \times_S \check{A}$ the morphism obtained from the closed immersion $A_1 \to A$ by base change along $\check{\alpha}$ and $\pi \colon A_1 \times_S \check{A} \to \check{A}$ the second projection. \Cref{Remark:AlternativeDescriptionPushForwardAtiyahExtension} (applied to the abelian scheme $A \times_S \check{A}$ over $\check{A}$) furnishes an isomorphism of sequences of $\cO_{\check{A}}$-modules
\[ (e, \id_{\check{A}})^\ast (\At_q(\cL)) \stackrel{\lowsim}{\too} \pi_\ast (\iota^\ast \cL).  \]
On the other hand, by \Cref{Prop:VanishingAtiyahClass}\eqref{p:vac-3}, the evaluation at the zero section
\[ (\cU_{\check{A}}) = q_\ast (\At_q(\cL)) \too (e, \id_{\check{A}})^\ast (\At_q(\cL)) \]
is an isomorphism of sequences of $\cO_{\check{A}}$-modules. Composing these isomorphisms 
yields an isomorphism of sequences of $\cO_{\check{A}}$-modules
\[ (\cU_{\check{A}}) \cong \pi_\ast (\iota^\ast \cL). \]
\end{remark}

\subsection{Moduli space of connections} \label{sec:ModuliSpaceConnections} 

\begin{definition}
The $S$-scheme $\check{A}^\natural = \bbP(\cU_{\check{A}}) \smallsetminus \bbP(\check{\alpha}^\ast \omega_{A})$ is called the \emph{moduli space of rank $1$ connections} on~$A$.
\end{definition}

\begin{theorem} \label{Thm:RepresentabilityOfModuliSpaceOfConnections} The $S$-scheme $\check{A}^\natural$ represents the functor associating with an $S$-scheme $S'$ the set of isomorphism classes of triples $(L, \phi, \nabla)$ made up of a homogeneous line bundle $(L, \phi)$ on the abelian scheme $A_{S'}$ and a connection  $\nabla \colon L \to \Omega^1_{A_{S'}/S'} \otimes L$.
\end{theorem}

\begin{proof} For an $S$-scheme $S'$, an $S'$-valued point of $\check{A}^\natural$ consists of the datum of a morphism of $S$-schemes $f \colon S' \to \check{A}$ and  a splitting $s \colon \cO_{S'} \to f^\ast \cU_{\check{A}}$ of the short exact sequence of $\cO_{S'}$-modules $f^\ast (\cU_{\check{A}})$. Let $(L, \phi)$ be the homogeneous line bundle on the abelian scheme $A' := A_{S'}$ obtained as the pull-back of the Poincar\'e bundle $\cL$ along the morphism $\id_A \times f \colon A' := A \times_S S' \to A \times_S \check{A}$. Let $\alpha' \colon A' \to S'$ be the morphism obtained from $\alpha$ by base change. Then, by \Cref{Prop:VanishingAtiyahClass}\eqref{p:vac-3}, the short exact sequence of $\cO_{A'}$-modules $\alpha'^\ast f^\ast (\cU_{\check{A}})$ is the Atiyah extension $(\At_{\alpha'}(L))$ of the homogeneous line bundle $L$.  By \Cref{Prop:CharacterizationConnections}, the splitting $\alpha'^\ast s \colon \cO_{A'} \to \At_{\alpha'}(L)$ corresponds to a connection on $L$ (notice the equality $\cEnd L = \cO_S$, due to $L$ being a line bundle).
\end{proof}

The tensor product of line bundles together with a connection endows $\check{A}^\natural$ with the structure of a group $S$-scheme.  The natural projection $\pi \colon \check{A}^\natural \to \check{A}$, $(L, \phi, \nabla) \mapsto (L, \phi)$ is a faithfully flat morphism of group $S$-schemes. The kernel of $\pi$ is by definition made of connections on the trivial bundle $\cO_A$. Now, a connection on the trivial line bundle is nothing but $\rd_{A/S} + \omega$ for a global differential form on $A$, where $\rd_{A/S}$ is the canonical $\alpha^{-1}\cO_{S}$-linear derivation. The isomorphism $\alpha_\ast \Omega_{A/S}^1 \cong e^\ast \Omega^1_{A/S} =: \omega_A$  thus yields a short exact sequence of group $S$-schemes
\[ 0 \too \bbV(\omega_{A}) \too \check{A}^\natural \too \check{A} \too 0.\]

The following remark will only be needed once,  quite further in text (see the proof of \Cref{Prop:UniversalCoverUniversalVectorExtension}). The reader may harmlessly skip it at first.

\begin{remark} \label{Rmk:GroupStructureViaRigidification} As any affine bundle on an abelian scheme, the group structure on $\check{A}^\natural$ is defined by a unique isomorphism of $\cO_{\check{A}}$-modules
\[ \check{\psi} \colon  \pr_1^\ast\cU_{\check{A}} +_\rB \pr_2^\ast \cU_{\check{A}} \too \check{\mu}^\ast \cU_{\check{A}},\]
where $\check{\mu}, \pr_1, \pr_2 \colon \check{A} \times_S \check{A} \to \check{A}$ are, respectively, the group law, the first and the second projection, and $+_\rB$ is the Baer sum of extensions (see Examples~\ref{Ex:AffineBundles} and~\ref{Ex:RigidificationAffineBundle} below).  According to \Cref{Ex:PushingForwardRigidificationLineBundle}, the isomorphism $\check{\psi}$ is the push-forward along the morphism $(A_1 \times_S \check{A}) \times_{A_1} (A_1 \times_S \check{A}) \to \check{A} \times_S \check{A}$
of the rigidification \[\check{\phi} \colon  (\id, \pr_1)^\ast \cL \otimes (\id, \pr_2)^\ast \cL  \too (\id, \check{\mu})^\ast \cL\] of the Poincar\'e bundle $\cL$ on the abelian scheme $A \times_S \check{A}$ over $A$.
\end{remark}

\subsection{The Lie algebra of the dual abelian scheme (redux)} \label{sec:LieAlgebraDualAbelianScheme} Let $\check{A}_1$ be the first-order thickening of $\check{A}$ along its zero section $\check{e}$, $\pi \colon A \times \check{A}_1 \to A$ the projection onto the first factor, $\iota \colon  A \times_S \check{A}_1 \to A \times_S \check{A}$ the morphism obtained from the closed embedding $\check{A}_1 \to \check{A}$ by base change with respect to $\alpha \colon A \to S$ and $\alpha_1 \colon A \times_S \check{A}_1 \to S$ the structural morphism.  Consider the isomorphism class
\[ [\iota^\ast \cL] \in \rH^{1}\left(A \times_S \check{A}_1, \cO_{A \times \check{A}_1}^\times\right)\]
of the line bundle $\iota^\ast \cL$ on $A \times_S \check{A}_1$, where $\cL$ is the Poincar\'e bundle on $A \times_S \check{A}$. Let $c_1(\iota^\ast \cL)$ be the global section of the sheaf of abelian groups $\rR^1 \alpha_{1 \ast} \cO_{A \times \check{A}_1}^\times$ on $S$ defined as the image of $[\iota^\ast \cL]$ via the homomorphism of abelian groups
\[ \rH^{1}\left(A \times_S \check{A}_1, \cO_{A \times \check{A}_1}^\times\right) \too \rH^0\left(S, \rR^1 \alpha_{1 \ast} \cO_{A \times \check{A}_1}^\times\right) \]
given by the Grothendieck--Leray spectral sequence
\[ \rH^{p}\left(S, \rR^q \alpha_{1\ast} \cO_{A \times \check{A}_1}^\times\right) \implies \rH^{p+q}\left(A \times_S \check{A}_1, \cO_{A \times \check{A}_1}^\times\right).\]
The first projection $\pi \colon A \times_S \check{A}_1 \to A$ induces a homeomorphism on the underlying topological spaces, whence an isomorphism $\rR^1 \alpha_\ast (\pi_\ast \cO_{A \times \check{A}_1}^\times) \cong \rR^1 \alpha_{1 \ast} \cO_{A \times \check{A}_1}^\times$ that will be treated as understood in what follows. The homomorphism $\cO_{A} \to \pi_\ast \cO_{A \times \check{A}_1}$ defines a splitting of 
%the short exact sequence 
$0 \to \alpha^\ast \omega_{\check{A}} \to \pi_\ast \cO_{A \times \check{A}_1} \to \cO_{A} \to 0$. It follows that the short exact sequence of sheaves abelian groups on $A$
\[ 0 \too \alpha^\ast \omega_{\check{A}} \stackrel{\exp}{\too}  \pi_\ast \cO_{A \times \check{A}_1}^\times \too \cO_A^\times \too 0\]
splits, where $\exp \colon v \mapsto 1 + v$ is the truncated exponential. We therefore have an exact sequence of sheaves of abelian groups on $S$
\[ 0 \too \rR^1 \alpha_\ast \alpha^\ast \omega_{\check{A}} \too \rR^1 \alpha_\ast \left(\pi_\ast \cO_{A \times \check{A}_1}^\times\right)\too \rR^1 \alpha_\ast \cO_A. \]
The pull-back of the line bundle $\iota^\ast \cL$ along the morphism $(\id, \check{e}) \colon A \to A \times_S \check{A}_1$ is trivial; thus $c_1(\iota^\ast \cL)$ is (the image via the truncated exponential of) a section of $\cO_S$-module $\rR^1 \alpha_\ast \alpha^\ast \omega_{\check{A}}$ still written $c_1(\iota^\ast \cL)$. The projection formula yields an isomorphisms
$ \rR^1 \alpha_\ast \alpha^\ast \omega_{\check{A}} \cong \rR^1 \alpha_\ast \cO_A \otimes \omega_{\check{A}} = \cHom(\Lie \check{A}, \rR^1 \alpha_\ast \cO_A)$ through which the class $c_1(\iota^\ast \cL)$ corresponds to an isomorphism (see \cite[Section~8.4, Theorem 1]{NeronModels})
\[ \Phi_A \colon \Lie \check{A} \too  \rR^1 \alpha_\ast \cO_A.\]

\subsection{Universal property of the canonical extension}  \label{sec:UniversalPropertyUniversalExtension} We borrow notation from \Cref{sec:LieAlgebraDualAbelianScheme}. For a vector bundle $F$ on $S$, the Grothendieck--Leray spectral sequence $\rH^{p}(S, \rR^q \alpha_\ast \alpha^\ast F) \Rightarrow \rH^{p+q}(A, \alpha^\ast F)$ yields an exact sequence of $\Gamma(S, \cO_S)$-modules
\begin{equation} \label{Eq:SESGrothendieckLeray} 0 \too \rH^1(S, F) \stackrel{\alpha^\ast}{\too} \rH^1(A, \alpha^\ast F) \stackrel{p_F}{\too} \rH^0(S, \rR^1 \alpha_\ast \alpha^\ast F).
\end{equation}
For an extension of $\cO_A$ by $\alpha^\ast F$,
\begin{equation} \tag{$\cF$}
0 \too \alpha^\ast F \too \cF \too \cO_A \too 0,
\end{equation}
let $[\cF] \in \rH^1(A, \alpha^\ast F)$ be its isomorphism class.  The global section $c(\cF) := p_F([\cF])$ of the $\cO_S$-module $\rR^1 \alpha_\ast \alpha^\ast F$ can also be seen as the connecting homomorphism in the long exact sequence of $\cO_S$-modules
\[ 0 \too F \too \alpha_\ast \cF \too \cO_S \stackrel{c(\cF)}{\too} \rR^1 \alpha_\ast \alpha^\ast F \too \cdots. \]
The abelian scheme $A$ coincides with the dual abelian scheme of $\check{A}$. By means of this identification, consider the canonical extension $(\cU_A)$ on $A$ and its isomorphism
\[ [\cU_A] \in \rH^1\left(A, \cHom\left(\cO_A, \alpha^\ast \omega_{\check{A}}\right)\right) = \rH^1\left(A, \alpha^\ast \omega_{\check{A}}\right).\]
The truncated exponential map $\exp \colon \alpha^\ast \omega_{\check{A}} \to \pi_\ast \cO_{A \times \check{A}_1}^\times$, $v \mapsto 1 + v$ induces a homomorphism of abelian groups
\[ \exp \colon \rH^1\left(A, \alpha^\ast \omega_{\check{A}}\right) \too \rH^1\left(A, \pi_\ast \cO_{A \times \check{A}_1}^\times\right).\]
On the other hand, the morphism $\pi$ is a homeomorphism on the underlying topological spaces, thus an isomorphism of abelian groups
\[ \rH^{1}\left(A \times_S \check{A}_1, \cO_{A \times \check{A}_1}^\times\right) \stackrel{\lowsim}{\too} \rH^1\left(A, \pi_\ast \cO_{A \times \check{A}_1}^\times\right).\]
The image of $[\iota^\ast \cL]$ via the preceding isomorphism is the isomorphism class $[\pi_\ast \iota^\ast \cL]$ of the invertible $\pi_\ast \cO_{A \times \check{A}_1}$-module $\pi_\ast \iota^\ast \cL$.

\begin{proposition} \label{Prop:ComputationIsomorphismClassCanonicalExtension} With the notation above,
\begin{align}
\tag{1} [\pi_\ast \iota^\ast \cL] &= \exp([\cU_A]),\label{p:cicce-1}\\
\tag{2} c(\cU_A) &= c_1(\iota^\ast \cL).\label{p:cicce-2}
\end{align}

\end{proposition}

The second statement can be reformulated by saying that, via $\rR^1 \alpha_\ast \alpha^\ast \omega_{\check{A}} \cong \cHom(\Lie \check{A}, \rR^1 \alpha_\ast \cO_A)$, the following equality holds:
\[ c(\cU_A) = \Phi_A. \]

\begin{proof} \eqref{p:cicce-1}~ \Cref{Rmk:AlternativeDescriptionCanonicalExtension} furnishes an isomorphism of short exact sequence of $\cO_A$-modules $(\cU_A) \cong \pi_\ast (\iota^\ast \cL)$. For an affine open cover $\{ A_i \}_{i \in I}$ of $A$, the isomorphism class $[\cU_A]$ is represented by a $1$-cocycle, say $f_{ij} \in \Gamma(A_i \cap A_j, \alpha^\ast \omega_{\check{A}})$ for $i, j \in I$. The invertible $\pi_\ast \cO_{A \times \check{A}_1}$-module $\pi_\ast \iota^\ast \cL$ is the glueing of the $\cO_{A_i}$-modules $(\pi_\ast \cO_{A \times \check{A}_1})_{\rvert A_i}$ along the transition maps $\exp(f_{ij}) = 1 + f_{ij}$. Relation~\eqref{p:cicce-2} follows immediately from \eqref{p:cicce-1}.
\end{proof}

For a homomorphism $\phi \colon \omega_{\check{A}} \to F$ with $F$ a vector bundle on $S$, let $(\cF_\phi)$ be the short exact sequence obtained by push-out of $(\cU_A)$ along the homomorphism $\alpha^\ast \phi$. Let $\gamma_F(\phi)$ denote its isomorphism class $[\cF_\phi] \in \rH^1(A, \alpha^\ast F)$; then this construction defines a map
\[ \gamma_F \colon \Hom(\omega_{\check{A}}, F) \too \rH^1(A, \alpha^\ast F).\]

\begin{theorem} \label{Th:UniversalPropertyUniversalExtension} The map $\gamma_F$ is injective and $\Gamma(S, \cO_S)$-linear, and its image is the set of isomorphism classes of extensions 
\begin{equation}
\tag{$\cF$} 0 \too \alpha^\ast F \too \cF \too \cO_A \too 0
\end{equation}
such that the short exact sequence of\, $\cO_S$-modules $e^\ast (\cF)$ splits.
\end{theorem}

This statement is an immediate consequence of the following more precise fact. To state it, consider the homomorphism $ e^\ast \colon \rH^1(A, \alpha^\ast F) \to \rH^1(S, F)$ given by the pull-back of an extension along the morphism $e$. Since $e$ is a section of $\alpha$, the composite map $e^\ast \circ \alpha^\ast$ is the identity of $\rH^1(S, F)$. Also, note that an extension $(\cF)$ as above splits if and only if $e^\ast [\cF] = 0$.

\begin{lemma} \label{Lemma:TechnicalLemmaUniversalPropertyCanonicalExtension} For a vector bundle $F$ on $S$,
\begin{enumerate}
\item\label{l:tlupce-1} $p_F \circ \gamma_F$ is the isomorphism induced by\, $\Phi_A \otimes \id_F \colon \cHom(\omega_{\check{A}}, F)  \to \rR^1 \alpha_\ast \alpha^\ast F$ on global sections, 
\item\label{l:tlupce-2} the following sequence of\, $\Gamma(S, \cO_S)$-modules is short exact:
\[ 0 \too \Hom(\omega_{\check{A}}, F) \stackrel{\gamma_F}{\too} \rH^1(A, \alpha^\ast F) \stackrel{e^\ast}{\too} \rH^1(S, F) \too 0.\]
\end{enumerate} 
\end{lemma}

\begin{proof} \eqref{l:tlupce-1}~ The diagram of $\cO_S$-modules
\[ 
\begin{tikzcd}
0 \ar[r] & \omega_{\check{A}} \ar[r] \ar[d, "\phi"] & \alpha_\ast \cU_A \ar[r] \ar[d] & \cO_S \ar[d, equal] \ar[r, "c(\cU_A)"] & \rR^1 \alpha_\ast \cO_A \otimes \omega_{\check{A}}  \ar[r] \ar[d, "\id \otimes \phi"] & \cdots \\
0 \ar[r] & F \ar[r] & \alpha_\ast \cF_\phi \ar[r] & \cO_S \ar[r, "c(\cF_\phi)"] & \rR^1 \alpha_\ast \cO_A \otimes F \ar[r]  & \cdots
\end{tikzcd}
\]
is commutative. It follows from \Cref{Prop:ComputationIsomorphismClassCanonicalExtension}\eqref{p:cicce-2} that the composite map $p_F \circ \gamma_F$ is the one induced on global sections by $\Phi_A \otimes \id_F \colon \Lie \check{A} \otimes F \to \rR^1 \alpha_\ast \cO_A \otimes F$. 

\eqref{l:tlupce-2}~ First of all, notice that the map $\gamma_F$ is injective and $\Gamma(S, \cO_S)$-linear. Second, the composite map $e^\ast \circ \gamma_F$ vanishes: for a homomorphism $\phi \colon \omega_{\check{A}} \to F$, the splitting of the short exact sequence of $\cO_S$-module $e^\ast (\cU_A)$ induced by the derivation $\rd_{\check{A}/S}$ (see \Cref{Rmk:SplittingCanonicalExtensionByCanonicalDerivation}) induces a splitting of $e^\ast (\cF_\phi)$. Therefore, it remains to show that the image of $\gamma_F$ is the whole $\ker(e^\ast)$. Set $r_F := (p_F \circ \gamma_F)^{-1} \circ p_F$. Then $r_F \circ \gamma_F$ is the identity of $\Hom(\omega_{\check{A}}, F)$, and the sequence of $\Gamma(S, \cO_S)$-modules
\[ 0 \too \rH^1(S, F) \stackrel{\alpha^\ast}{\too}  \rH^1(A, \alpha^\ast F)  \stackrel{r_F}{\too} \Hom(\omega_{\check{A}}, F) \too 0\]
is short exact by \eqref{Eq:SESGrothendieckLeray}. Since $\gamma_F$ and $e^\ast$ are sections of, respectively,  $r_F$ and $\alpha^\ast$, the result follows.
\end{proof}

\subsection{Preliminaries on extensions of an abelian scheme} Let $G$ be an affine, commutative, faithfully flat and finitely presented group $S$-scheme. 

 \subsubsection{Principal bundles} \label{Sec:PrincipalBundles}  A \emph{principal $G$-bundle} on $A$ is a faithfully flat $A$-scheme $P$ endowed with an action of $G$ such that the morphism $ (\sigma, \pr_P) \colon G \times_S P \to P \times_A P $ is an isomorphism, where $\sigma, \pr_P \colon G \times_S P \to P$ are, respectively, the morphism defining the action and the projection onto $A$. Let $\rH^1_{\textup{fppf}}(A, G)$ denote the set of isomorphism classes of principal $G$-bundles on $A$. 
 
 Let $\rho \colon G \to G'$ be a morphism of group $S$-schemes, where $G'$ has the same properties of $G$, and let $P$ be a principal $G$-bundle on $A$. 
 The quotient $\rho_\ast P$ of $S$-scheme $G' \times_S P$ via the action $g(g',x) = (\rho(g)g', gx)$ of $G$ exists and is a principal $G'$-bundle on $A$, called the \emph{push-out} (see \cite[Propositions 4.5.6 and 12.1.2]{OlssonStacks}).
 For principal $G$-bundles $P$ and $P'$ on $A$, the push-out 
 \[ P \wedge P' := \mu_{G \ast} (P \times_A P')\]
 of the principal $G \times_S G$-bundle $P \times_A P'$ along the sum map $\mu_G \colon G \times_S G \to G$ is called the \emph{sum}. This operation endows the set $\rH^1_{\textup{fppf}}(A, G)$ with the structure of an abelian group with neutral element $G \times_S A$ and inverse $P \mapsto [-1]_\ast P$, where $[-1]$ is the inverse map on $G$. Furthermore, the push-out of an endomorphism of $G$ as a group $S$-scheme induces an endomorphism of the abelian group $\rH^1_{\textup{fppf}}(A, G)$. This equips $\rH^1_{\textup{fppf}}(A, G)$ with the structure of a module on the ring $\End(G)$ of endomorphisms of $G$. For instance, if $G = \bbV(F)$ for some vector bundle $F$ on $S$, the set $\rH^1_{\textup{fppf}}(A, G)$ is naturally a $\Gamma(S, \cO_S)$-module.

\begin{example} \label{Ex:LineBundlesAsPrincipalBundles} Suppose $G = \Gm$. The principal $\Gm$-bundle associated with a line bundle $L$ on $A$ is the total space of $L$ deprived of its zero section $\bbV(L)^\times$. The  so-defined map $\Pic(A) \to \rH^1_{\textup{fppf}}(A, \Gm)$, $L \mapsto  \bbV(L)^\times$ is an isomorphism of abelian groups (see \cite[Expos\'e XI, Proposition 5.1]{SGA1}), that is, a sum of principal $\Gm$-bundles corresponds to a tensor product of line bundles.
\end{example}

\begin{example} \label{Ex:AffineBundles} Suppose $G = \bbV(F)$ for some vector bundle $F$ on $S$, and let 
  \begin{equation} \tag{$\cF$} 0 \too \alpha^\ast F \too \cF \stackrel{p}{\too} \cO_A \too 0
  \end{equation}
  be a short exact sequence of $\cO_A$-modules. The $A$-scheme $\bbA(F) := \bbP(\cF) \smallsetminus \bbP(\alpha^\ast F)$ is a principal $\bbV(F)$-bundle. The map $\rH^1(A, \alpha^\ast F) \to \rH^1_{\textup{fppf}}(A, \bbV(F))$, $(\cF) \mapsto \bbA(\cF)$ is an isomorphism of $\Gamma(S, \cO_S)$-modules (see \cite[Expos\'e XI, Proposition 5.1]{SGA1}). In particular, the sum of such principal $\bbV(F)$-bundles is the $\bbV(F)$-bundle associated with
  the Baer sum of the corresponding extensions.
\end{example}

\subsubsection{Law groups on principal bundles} \label{LawGroupsPrincipalBundles} An \emph{extension} of $A$ by $G$ is the datum of a short exact sequence of commutative group $S$-schemes
\begin{equation} \tag{$E$} 0 \too G \stackrel{i_E}{\too} E \stackrel{p_E}{\too} A \too 0,
\end{equation}
where the morphism $p_E$ is faithfully flat.  Note that an extension $E$ of $A$ by $G$ is naturally a principal $G$-bundle over $A$. Since the natural homomorphism $\cO_S \to f_\ast \cO_A$ is an isomorphism, the commutativity of $E$ is automatic. Morphisms of extensions are defined in the evident way. An isomorphism of extensions is a morphism inducing the identity on $A$ and $G$. Let $\Ext(A, G)$ be set of isomorphism classes of extensions of $A$ by $G$. For a morphism of $S$-group schemes $\rho \colon G \to G'$, with $G'$ having the same properties as $G$, and an extension $(E)$, the cokernel $\phi_\ast (E)$ of the morphism $(i_E, \rho) \colon G \to E \times_S G'$ is called the \emph{push-out} of $(E)$.  The \emph{Baer sum} of extensions $(E)$ and $(E')$ of $A$ by $G$ is the push-out of $E \times_A E$ along the sum morphism $G \times_S G \to G$. The Baer sum endows $\Ext(A, G)$ with the structure of an abelian group.  Similarly to the case of principal $G$-bundles, the abelian group $\Ext(A, G)$ is endowed with the structure of an $\End(G)$-module. Seeing an extension as a principal bundle gives rise to a homomorphism of $\End(G)$-modules
\[ \lambda_G \colon \Ext(A, G) \too \rH^1_{\textup{fppf}}(A, G). \]
A \emph{rigidification} of a principal $G$-bundle $P$ over $A$ is an isomorphism 
\[ \phi \colon \pr_1^\ast P \wedge \pr_2^\ast P \too  \mu_A^\ast P\] of principal $G$-bundles over $A \times_S A$.\footnote{The definition of a rigidification in \cite[Expos\'e VII, Section~1]{SGA7I} involves the commutativity of two diagrams which is automatic over abelian schemes.} The map $\lambda_G$ is injective, and its image is the set of isomorphism classes of principal $G$-bundles admitting a rigidification (see \cite[Th\'eor\`eme 15.5]{SerreGroupesAlgebriques}).

\begin{example} Suppose $G = \Gm$ and identify principal $\Gm$-bundles with line bundles as in \Cref{Ex:LineBundlesAsPrincipalBundles}. A rigidification of a line bundle $L$ on $A$ is an isomorphism of $\cO_{A \times_S A}$-modules $\phi \colon \pr_1^\ast L \otimes \pr_2^\ast L  \to  \mu_A^\ast L $.
\end{example}

\begin{example} \label{Ex:RigidificationAffineBundle} Suppose $G = \bbV(F)$ for some vector bundle $F$ on $S$. Identify principal $\bbV(F)$-bundles with extensions of $\cO_A$ by $f^\ast F$ as in \Cref{Ex:AffineBundles}. For a short exact sequence of $\cO_A$-modules
\begin{equation} \tag{$\cF$} 0 \too f^\ast F \too \cF \too \cO_A \too 0, \end{equation}
a rigidification of $(\cF)$ is an isomorphism of short exact sequences of $\cO_{A \times_S A}$-modules
\[ \phi \colon  \pr_1^\ast (\cF) +_\rB \pr_2^\ast (\cF) \stackrel{\lowsim}{\too}  \mu_A^\ast (\cF),\]
where $\pr_1^\ast (\cF) +_\rB \pr_2^\ast (\cF)$ is the Baer sum of the extensions $\pr_1^\ast (\cF)$ and $\pr_2^\ast (\cF)$.
\end{example}

\begin{example} \label{Ex:PushingForwardRigidificationLineBundle} For a vector bundle $F$  on $S$, let $S'$ be the first-order thickening of $\bbV(F^\vee)$ along its zero section. Let $L$ be a line bundle on $A':= A \times_S S'$ endowed with a rigidification $\phi \colon \pr_1^\ast L \otimes \pr_2^\ast L \to \mu_{A'}^\ast L$,  with the obvious notation.  Let $s \colon A \to A'$ be the closed immersion induced by the zero section of $F^\vee$ and $\pi \colon A' \to A$ the projection onto $A$. The $\cO_A$-module $\cF := \pi_\ast L$ sits into the following short exact sequence of $\cO_A$-modules:
\begin{equation} \tag{$\cF$}
0 \too f^\ast F \otimes s^\ast L \too \cF \too s^\ast L \too 0.
\end{equation}
According to \Cref{Prop:TensorProductBecomesBaerSum}, the push-forward of the $\cO_{A' \times A'}$-module $\pr_1^\ast L \otimes \pr_2^\ast L$ along the morphism  $ \pi \times \pi \colon A' \times_{S'} A' \to A \times_S A $
 is the Baer sum of the extensions $\pr_1^\ast \cF \otimes \pr_2^\ast s^\ast L$ and $\pr_1^\ast s^\ast L \otimes \pr_2^\ast \cF $.  Assume further that the line bundle $s^\ast L$ is isomorphic to $\cO_A$. Then, the vector bundle $\cF$ is an extension of $\cO_A$ by $f^\ast F$, and the isomorphism $(\pi \times \pi)_\ast \phi \colon \pr_1^\ast \cF +_\rB \pr_2^\ast \cF  \to \mu_A^\ast \cF$ is a rigidification like the one considered in Example~\ref{Ex:RigidificationAffineBundle}.
\end{example}

\subsection{The universal vector extension} A \emph{vector extension} of the abelian scheme $A$ is an extension of $A$ by $\bbV(F)$ for some vector bundle $F$ on $S$ called its \emph{vector part}. Recall the injective homomorphism of $\Gamma(S, \cO_S)$-modules 
\[ \lambda_F := \lambda_{\bbV(F)} \colon \Ext(A, \bbV(F)) \too \rH^1_{\textup{fppf}}(A, \bbV(F)) = \rH^1(A, \alpha^\ast F)\]
defined in \Cref{LawGroupsPrincipalBundles} and the map $p_F \colon \rH^1(A, \alpha^\ast F) \to \rH^0(S, \rR^1 \alpha_\ast \alpha^\ast F)$ introduced in \eqref{Eq:SESGrothendieckLeray}. According to \cite[Proposition 1.10]{MazurMessing}, the composite homomorphism
\[ \lambda_{F/S} := p_F \circ \lambda_F \colon  \Ext(A, F) \too \rH^0(S, \rR^1 \alpha_\ast \alpha^\ast F) \]
is bijective. For the vector bundle $F = (\rR^1 \alpha_\ast \cO_A)^\vee$ on $S$, this gives an isomorphism
\[ \Ext(A, \bbV((\rR^1 \alpha_\ast \cO_A)^\vee)) \cong  \End \rR^1 \alpha_\ast \cO_A.\]
\begin{definition} The universal vector extension is the one corresponding to the identity via the above isomorphism:
\begin{equation} \tag{$\rE(A)$}
0 \too \bbV((\rR^1 \alpha_\ast \cO_A)^\vee) \too \rE(A) \too A \too 0.
\end{equation}
\end{definition}

The extension $\rE(A)$ deserves the `universal' title because of the following property.  First notice that, for vector bundles $F$ and $F'$ on $S$ and a homomorphism of $\cO_S$-modules $\phi \colon F \to F'$, the diagram of $\Gamma(S, \cO_S)$-modules
\[
\begin{tikzcd}
\Ext(A, F)  \ar[r, "\lambda_{F/S}"] \ar[d]&\rH^0(S, \rR^1 \alpha_\ast \alpha^\ast F) \ar[d, "\id \otimes \phi"]\\
\Ext(A, F')  \ar[r, "\lambda_{F'/S}"]&\rH^0(S, \rR^1 \alpha_\ast \alpha^\ast F')
\end{tikzcd}
\]
is commutative, where the leftmost vertical arrow is given by push-out along $\phi$. Moreover, for a vector extension $G$, the global section $\lambda_{F/S}(G)$ of $\rR^1 \alpha_\ast \alpha^\ast F$ defines a homomorphism of $\cO_S$-modules $\lambda_{F/S}(G) \colon (\rR^1 \alpha_\ast \cO_A)^\vee \to F$. The extension $G$ is isomorphic (in a unique way!) to the push-out of $(\rE(A))$
along $\lambda_{F/S}(G)$.  The moduli space $A^\natural$ of rank $1$ connections on $\check{A}$ is a vector extension of the abelian scheme $A$ with vector part $\omega_{\check{A}}$. Recall the isomorphism of $\cO_S$-modules $\Phi_A \colon \Lie \check{A} \to \rR^1 \alpha_\ast \cO_A$ introduced in \Cref{sec:LieAlgebraDualAbelianScheme}. \Cref{Prop:ComputationIsomorphismClassCanonicalExtension}\eqref{p:cicce-2} implies the following. 

\begin{theorem} \label{Thm:CanonicalIsomorphismUniversalVectorExtensionModuliSpaceConnections} The vector extension $A^\natural$ is the push-out of\, $\rE(A)$ along
\[ \Phi_A^\vee \colon (\rR^1 \alpha_\ast \cO_A)^\vee \stackrel{\lowsim}{\too} (\Lie \check{A})^\vee = \omega_{\check{A}}.\]
\end{theorem}

As a consequence, one obtains a down-to-earth description of how vector extensions are constructed. For a homomorphism of $\cO_S$-modules $\phi \colon \omega_{\check{A}} \to F$, let $E_\phi$ be the push-out of $A^\natural$ along $\phi$. 

\begin{corollary}\label{Prop:ModuliSpaceConnectionIsUniversalExtension} The map
$\Hom(\omega_{\check{A}}, F) \to\Ext(A, F)$, $\phi \mapsto [E_\phi]$ 
is the inverse of $(\Phi_{A}^\vee \otimes \id _F) \circ \lambda_{F / S}$.
\end{corollary}

Observe that the affine bundle underlying $E_\phi$ is $\bbP(\cF_\phi) \smallsetminus \bbP(\alpha^\ast F)$, where $(\cF_\phi)$ is the push-out of $(\cU_{A})$ along $\phi$.

\section{Preliminaries on Tate--Raynaud uniformization}  \label{sec:TateRaynaudUniformization}

Let $K$ be a complete non-trivially valued non-Archimedean field and $R$ its ring of integers.

\subsection{Raynaud's generic fiber of a formal abelian scheme} Let $\cS$ be an admissible formal $R$-scheme, $\cB$  a formal abelian scheme over the formal $R$-scheme $\cS$, $\check{\cB}$ its dual  and $\cL_{\cB}$ the Poincar\'e bundle on $\cB \times_{\cS} \check{\cB}$. As customary, Raynaud's generic fibers of formal schemes are referred to by straight letters (as opposed to curly ones for formal schemes). More explicitly, let $S$, $B$ and $\check{B}$ be  Raynaud's generic fibers of $\cS$, $\cB$ and $\check{\cB}$. Let $\cL_{B}$ be the line bundle on $B \times_S \check{B}$ deduced from $\cL_{\cB}$. The $K$-analytic space $\check{B}$ represents the functor associating with a $S$-analytic space the group of isomorphism classes of homogeneous line bundles on $B \times_S S'$.  Moreover, the universal object is the line bundle $\cL_B$ on $B \times_S \check{B}$ (see \cite[Proposition 6.2]{BoschLutkebohmertDegenerating}). Let $\beta \colon B \to S$ and $\check{\beta} \colon \check{B} \to S$ be the structural morphisms. For an $S$-analytic space $S' \to S$ and morphisms $b \colon S' \to B$, $\check{b} \colon S' \to \check{B}$ of $S$-analytic spaces, let 
\[ \cL_{B, (b, \check{b})} := (b, \check{b})^\ast \cL_B\]
be the line bundle on $S'$ obtained by pulling back the Poincar\'e bundle $\cL_B$ on $B \times_S \check{B}$ along the morphism $(b, \check{b}) \colon S' \to B \times_S \check{B}$.

\subsection{Datum of a toric extension} Let $\check{\Lambda}$ be a free abelian group of rank equal to that of $\Lambda$, $\check{\Lambda}_S$ the constant $S$-analytic group with value $\check{\Lambda}$ and $T$ the split $S$-torus with group of characters $\check{\Lambda}$. From now on suppose that $S$ is connected, so that sections of the morphism $\check{\Lambda}_S \to S$ are naturally in one-to-one correspondence with elements $\check{\Lambda}$ and will be henceforth identified with those. Let  $\check{c} \colon \check{\Lambda}_S \to \check{B}$ be a morphism of $S$-analytic groups, and consider the extension $\epsilon \colon E \to S$ of the proper $K$-analytic group $B$ by the torus $T$ determined by $\check{c}$:
\[ 0 \too T \too E \stackrel{p}{\too} B \too 0. \]

The extension $E$ is described as follows. For an $S$-analytic space $S'$, an $S'$-valued point $g$ of the $S$-analytic space $E$ is the datum of 
\begin{itemize}
\item an $S'$-valued point $b = p(g)$ of  $B$ and,
\item for $\check{\chi} \in \check{\Lambda}$, a trivialization $ \langle g, \check{\chi} \rangle_E$ of the line bundle $\cL_{B, (b, \check{c}(\check{\chi}))}$ on $S'$.
\end{itemize}
 Moreover, the trivializations are required to satisfy the following compatibility: for $\check{\chi}, \check{\chi}' \in \check{\Lambda}$, 
\[\left\langle g, \check{\chi}  +  \check{\chi}' \right\rangle_E = \left\langle g, \check{\chi} \right\rangle_E \otimes \left\langle g, \check{\chi}' \right\rangle_E,\]
where the equality is meant to be understood via the isomorphism
\begin{equation} \label{Eq:IsomorphismEivingMultiplication}
   \cL_{B, (b, \check{c}(\check{\chi}))} \otimes  \cL_{B, (b, \check{c}(\check{\chi}'))} \cong  \cL_{B, (b, \check{c}(\check{\chi} + \check{\chi}'))} 
\end{equation}
induced by the implied rigidification of the homogeneous line bundle $\cL_B$.

\subsection{Dual datum} Let $\Lambda$ be a free abelian group of finite rank, $\Lambda_S$ the constant $S$-analytic group with value $\Lambda$, $\check{T}$ the split $S$-torus with group of characters $\Lambda$ and $i \colon \Lambda_S \to E$ a morphism of $S$-analytic groups which is a closed immersion.  See the abeloid variety $B$ as the dual of $\check{B}$; then the group morphism $c = p \circ i \colon M_S \to B$ determines an extension $\check{E}$ of the proper $S$-analytic group $\check{B}$ by the torus $\check{T}$:
 \[ 0 \too \check{T} \too \check{E} \stackrel{\check{p}}{\too} \check{B} \too 0. \]
 
Just to fix notation, for an $S$-analytic space $S'$, an $S'$-valued point $\check{g}$ of $\check{E}$ corresponds to the datum of
\begin{itemize}
\item an $S'$-valued point $\check{b} = \check{p}(\check{g})$ of  $\check{B}$ and,
\item for $\chi \in \Lambda$, a trivialization $ \langle \chi, \check{g} \rangle_{\check{E}}$ of the line bundle $\cL_{B, (c(\chi), \check{b})}$ on $S'$.
\end{itemize}
As before, the trivializations are subsumed to the relation, for $\chi, \chi' \in \Lambda$,
\[ \langle \chi + \chi', \check{g} \rangle_{\check{E}} = \langle \chi, \check{g} \rangle_{\check{E}} \otimes \langle \chi', \check{g} \rangle_{\check{E}} .\]
Now, for $\check{\chi} \in \check{\Lambda}$, by the symmetry of the Poincar\'e bundle, there is a unique  $S$-valued point $\check{\imath}(\check{\chi})$ of $\check{E}$ such that, for $\chi \in \Lambda$,
\[ \left\langle \chi, \check{\imath}(\check{\chi}) \right\rangle_{\check{E}} = \left\langle i(\chi), \check{\chi} \right\rangle_E.\]
This defines an injective morphism of $S$-analytic groups $\check{\imath} \colon \check{\Lambda}_S \to \check{E}$. 

In an attempt of unburdening the (already overwhelming) notation, in what follows $\Lambda_S$ is identified with the image of $i$, and the subscript $E$ is dropped from the pairing $\langle -, -\rangle_E$, and similarly for $\check{\Lambda}_S$, $\check{\imath}$ and $\langle -, -\rangle_{\check{E}}$.

\subsection{Quotient} \label{sec:QuotientAbeloidVariety}
Since the subgroup $\Lambda_S$ is closed (hence fiberwise discrete) in $E$, the (topological) quotient $A := E / \Lambda_S$ exists. Moreover, assume that the structural morphism $\alpha \colon A \to S$ is proper. Under these working hypotheses, the bilinear pairing $\langle -, - \rangle$ on $\Lambda \times \check{\Lambda}$ is non-degenerate, the subgroup  $\check{\Lambda}_S$ is closed in $\check{E}$, the quotient $\check{A} := \check{E} / \check{\Lambda}_S$ exists, and the structural morphism $\check{\alpha} \colon \check{A} \to S$ is proper (see \cite[Proposition 3.4]{BoschLutkebohmertDegenerating}). Let $u \colon E \to A$ and $\check{u} \colon \check{E} \to \check{A}$ be the quotient maps. The situation is summarized in the following diagrams:
\begin{center}
\begin{tikzcd}
& \Lambda_S \ar[d, "i"] \ar[dr, "c"] & \\
T \ar[r] & E \ar[r, "p"] \ar[d, "u"] & B \\
& A\rlap{,}
\end{tikzcd} \hspace{60pt}
\begin{tikzcd}
& \check{\Lambda}_S \ar[d, "\check{\imath}"] \ar[dr, "\check{c}"] & \\
\check{T} \ar[r] & \check{E} \ar[r, "\check{p}"] \ar[d, "\check{u}"] & \check{B} \\
& \check{A}\rlap{.}
\end{tikzcd}
\end{center}

\subsection{Coherent sheaves on the quotient} Descent of modules along the morphism $u$ can be restated in terms of $\Lambda_S$-linearizations. More precisely, the datum of a coherent $\cO_A$-module is equivalent to that of a coherent $\cO_E$-module $V$ endowed with a $\Lambda_S$-linearization 
\[\lambda \colon \pr_E^\ast V \too \sigma^\ast V, \]
where $\sigma, \pr_E \colon \Lambda_S \times_S E \to E$ are, respectively, the restriction to $\Lambda_S \times_S E$ of the group law of $E$ and the projection onto the second factor. Quite concretely, the group $\Lambda_S$ being constant, the datum of a $\Lambda_S$-linearization of a coherent $\cO_E$-module $V$ boils down to that, for $\chi \in \Lambda$, of an isomorphism of $\cO_E$-modules
\[ \lambda_\chi \colon V \too \tr_\chi^\ast V,\]
where $\tr_\chi$ is the translation by $\chi$ on $E$. Additionally, the collection of isomorphisms $\lambda_\chi$ is required to fulfil the following compatibility, for $\chi, \chi' \in \Lambda$:
\[ \lambda_{\chi + \chi'} = \tr_\chi^\ast \lambda_{\chi'} \circ \lambda_\chi.\]

\subsection{Homogeneous line bundles}  \label{sec:HomogeneousLineBundlesOnTheUniversalCover} Owing to \cite[Theorem 6.7]{BoschLutkebohmertDegenerating},  given a homogeneous line bundle $L$ on $A$, its pull-back $u^\ast L$ on $E$ is isomorphic to $p^\ast M$ for some line bundle $M$ on $B$. Moreover, the natural rigidification 
\[ \pr_1^\ast p^\ast M \otimes \pr_2^\ast p^\ast M \too \mu_E^\ast p^\ast M,\]
where $\mu_E$ is the group law on $E$, is equivariant with respect to the $\Lambda_S$-linearization $\lambda$ on $p^\ast M$ induced by the isomorphism $u^\ast L \cong p^\ast M$. It follows that the $\Lambda_S$-linearization $\lambda$ can be expressed as the datum of isomorphisms, for $\chi \in \Lambda$,
\[ \lambda_\chi \colon p^\ast M \too \tr_\chi^\ast p^\ast M, \quad s \longmapsto s \otimes r(\chi),\]
where the isomorphism $\tr_\chi^\ast p^\ast M \cong p^\ast M \otimes \epsilon^\ast c(\chi)^\ast M$ coming from the homogeneity of $M$ is taken into account (recall that $\epsilon$ is the structural morphism of $E$) and $r$ is a trivialization of the line bundle $c^\ast M$ on $\Lambda_S$. What is more, the trivialization $r$ must satisfy, for $\chi, \chi' \in \Lambda$, the relation
\[ r(\chi) \otimes r(\chi') = r(\chi + \chi'),\]
where, as is customary at this stage, the above equality is meant to be understood via the isomorphism $c(\chi)^\ast M \otimes c(\chi')^\ast M \cong c(\chi + \chi')^\ast M$ given by the rigidification of $M$.

\subsection{Duality} \label{sec:DualityAbeloidVariety}  By \cite[Theorem 6.8]{BoschLutkebohmertDegenerating} the $S$-analytic group $\check{A}$ represents the functor associating with an $S$-analytic space $S'$ the set of isomorphism classes of homogeneous line bundles on $A \times_S S'$. Furthermore, let $\cL_A$ be the Poincar\'e bundle   on $A \times_S \check{A}$. According to \emph{loc.~cit.}, there is a (necessarily unique) isomorphism of line bundles
\[ \xi \colon (u, \check{u})^\ast \cL_A \stackrel{\lowsim}{\too} \cL_E:= (p, \check{p})^\ast \cL_B \]
on $E \times_S \check{E}$ compatible with the implied rigidifications of the homomogeneous line bundles $\cL_A$ and $\cL_B$. The isomorphism $\xi$ endows the line bundle $\cL_E$ with a $(\Lambda \times \check{\Lambda})_S$-linearization $\lambda$, which can be described as follows. For characters $\chi \in \Lambda$ and $\check{\chi} \in \check{\Lambda}$, an $S$-analytic space $S'$ and $S'$-valued points $x$ of $E$ and $\check{x}$ of $\check{E}$, the isomorphism of $\cO_{S'}$-modules
\[ \lambda_{(\chi, \chi'), (x, x')} \colon \cL_{E, (x, \check{x})} \stackrel{\lowsim}{\too} \cL_{E, (x + \chi, \check{x} + \check{\chi})}\]
induced by the linearization $\lambda$ is
\begin{equation} \label{Eq:LinearizationPoincareBundle}
v \longmapsto \left(\left\langle x, \check{\chi}\right\rangle_E \otimes \left\langle \chi, \check{\chi}\right\rangle_{E}\right) \otimes \left( \left\langle \chi, \check{x} \right\rangle_{\check{E}} \otimes v\right).
\end{equation}
In order to make sense of \eqref{Eq:LinearizationPoincareBundle}, observe that $\langle x, \check{\chi}\rangle_E$ is by definition a section of the line bundle $\cL_{E, (x, \check{\chi})}$, while $\langle \chi, \check{\chi}\rangle_{E}$ is a section of the line bundle $\cL_{E, (\chi, \check{\chi})}$ and $\langle x, \check{\chi}\rangle_E \otimes \langle \chi, \check{\chi}\rangle_{E}$ is seen as a section of the line bundle $\cL_{E, (x + \chi, \check{\chi})}$ via the isomorphism of $\cO_{S'}$-modules
\[  \cL_{E, (x, \check{\chi})} \otimes \cL_{E, (\chi, \check{\chi})} \cong \cL_{E, (x + \chi, \check{\chi})} \]
induced by the rigidification of the homogeneous line bundle $\cL_B$. Arguing similarly, $\langle \chi, \check{x} \rangle_{\check{E}} \otimes v$ is a section of line bundle $\cL_{E, (x + \chi, x')}$, so that the right-hand side of \eqref{Eq:LinearizationPoincareBundle} is the section of the line bundle $\cL_{E, (x + \chi, \check{x} + \check{\chi})}$ via the isomorphism of $\cO_{S'}$-modules
\[  \cL_{E, (x + \chi, \check{\chi})} \otimes \cL_{E, (x + \chi, \check{x})} \cong  \cL_{E, (x + \chi, \check{x} + \check{\chi})}. \]

\section{The universal vector extension of an abeloid variety} 

Let $K$ be a non-trivially valued complete non-Archimedean field and $S$ a $K$-analytic space.

\subsection{The canonical extension} 

\subsubsection{Definition} Let $A$ be an abeloid variety over $S$, that is, a proper and smooth $S$-analytic group with (geometrically) connected fibers. Suppose that the functor associating with an $S$-analytic space $S'$ the set of isomorphism classes of (fiberwise) homogeneous line bundles on $A_{S'}$ is represented by an abeloid variety $\check{A}$ over $S$. For instance, this is the case if $S$ is a $K$-rational point (see \cite[Corollary 7.6.5]{LutkebohmertCurves}) or if $\check{A}$ admits a uniformization as the one described in \Cref{sec:TateRaynaudUniformization} (see \cite[Theorem 6.8]{BoschLutkebohmertDegenerating}).  Under this assumption, translating the arguments of \Cref{Sec:UniversalExtensionAbelianScheme} into the rigid analytic framework permits one to define the canonical extension $(\cU_A)$ on the abeloid variety $A$ and the moduli space $A^\natural$ of rank $1$ connections on $A$.

\subsubsection{Canonical extension on the universal cover} From now on, and up until \Cref{sec:UniversalCoverAffineBundles}, the abeloid variety $A$ is supposed to admit a uniformization as the one described in \Cref{sec:TateRaynaudUniformization}. We borrow notation  introduced therein and consider the canonical extension $(\cU_B)$ on (Raynaud's generic fiber of) the formal abelian scheme $B$.

\begin{definition} The pull-back of differential forms along the morphism $\check{p} \colon \check{E} \to \check{B}$ induces a homomorphism of $\cO_S$-modules $\rd \check{p} \colon \omega_{\check{B}} \to \omega_{\check{E}}$.  The short exact sequence of $\cO_E$-modules
\begin{equation}
\tag{$\cU_E$} 0 \too \epsilon^\ast \omega_{\check{E}} \too \cU_{E} \too \cO_E \too 0
\end{equation}
obtained as the push-out of short exact sequence of $\cO_E$-modules $p^\ast (\cU_B)$ along $\rd \check{p}$ is called the \emph{canonical extension} of $\cO_E$.
\end{definition}

\subsubsection{Alternative description} \label{sec:AlternativeDescriptionCanonicalExtensionUniversalCover}To perform `explicit' computations, it is often useful to have at hand a down-to-earth expression for $\cU_E$, similar to that in \Cref{Rmk:AlternativeDescriptionCanonicalExtension}. Let $X = A$, $B$, $E$ and, respectively, $\check{X} = \check{A}$, $\check{B}$, $\check{E}$. Let $\check{X}_1 = \check{A}_1$, $\check{B}_1$, $\check{E}_1$ be the first-order thickenings of $\check{X}$, and
\begin{align*}
\iota_X \colon X \times_S \check{X}_1 \too X \times_S \check{X}, \quad \pi_X \colon X \times_S \check{X}_1 &\too X, 
\end{align*}
respectively, the closed immersion and the projection onto the first factor. With this notation, the considerations in \Cref{Rmk:AlternativeDescriptionCanonicalExtension} furnish isomorphisms
\begin{align*}
\cU_A \cong \pi_{A \ast} \iota^\ast_A \cL_A, \quad \cU_B \cong \pi_{B \ast} \iota^\ast_B \cL_B,
\end{align*}
where $\cL_A$ and $\cL_B$ are, respectively, the Poincar\'e bundles on $A \times_S \check{A}$ and $B \times_S \check{B}$.  Recall the isomorphism 
\[ \xi \colon (u, \check{u})^\ast \cL_A \stackrel{\lowsim}{\too} \cL_E := (p, \check{p})^\ast \cL_B \]
of line bundles on $E \times_S \check{E}$ considered in \Cref{sec:DualityAbeloidVariety}. By definition, the canonical extension $\cU_E$ is the push-forward along the morphism $\pi_E$ of the line bundle $\iota^\ast_E \cL_E$ on $E \times_S \check{E}_1$:
\[ \cU_E = \pi_{E \ast} \iota^\ast_E \cL_E.\]
Pushing forward the isomorphism $\iota_E^\ast \xi$ along the map $\pi_E$ yields an isomorphism of $\cO_E$-modules
\[ \pi_{E \ast} \iota^\ast_E \xi \colon  u^\ast \cU_{A} \stackrel{\lowsim}{\too} \cU_E.\]

\subsection{Linearization of the canonical extension} \label{Sec:LinearizationCanonicalExtension} By means of the isomorphism $u^\ast \cU_{A} \cong \cU_E$ described above, the canonical extension $\cU_E$ acquires a natural $\Lambda_S$-linearization. The task undertaken here is to give an explicit expression for it. Unfortunately, this point is as crucial as dreadfully technical.

\subsubsection{} \label{sec:LinearizationPoincareBundleThickeningUniversalCover} Describing the $\Lambda_S$-linearization of the line bundle $\iota_E^\ast \cL_E$ is easily achieved. Indeed, such a linearization is the pull-back along the morphism $\iota_E$ of the $(\Lambda \times \check{\Lambda})_S$-linearization of the line bundle $\cL_E$. Let $\check{\jmath} \colon \check{E}_1 \to \check{E}$ denote the closed immersion and, for $\chi \in \Lambda$, consider the trivialization $\langle \chi, \check{\jmath} \rangle$ of the line bundle 
\[
\cL_{E, (\chi, \check{\jmath}\,)} = \left(\chi, \id_{\check{E}_1}\right)^\ast \iota_E^\ast \cL_E\]
 on $\check{E}_1$. Evaluating \eqref{Eq:LinearizationPoincareBundle} at $\check{\chi} = 0$, $x = \id_E$ and $\check{x} = \check{\jmath}$ shows that the $\Lambda_S$-linearization of the line bundle $\iota_E^\ast \cL_E$ is given, for $\chi \in \Lambda$, by the isomorphism
\begin{equation} \label{Eq:LinearizationPoincareBundleOnThickening}
\iota_E^\ast \cL_E \too \left(\tr_\chi, \id_{\check{E}_1}\right)^\ast \iota_E^\ast \cL_E, \quad v \longmapsto  v \otimes \langle \chi, \check{\jmath}\,\rangle,
\end{equation}
where $\tr_\chi$ is the translation by $\chi$ on $E$. To make sense of the formula, notice that the homogeneity of the line bundle $\iota_E^\ast \cL_E$ furnishes an isomorphism
\begin{equation} \label{Eq:TranslationPoincareBundleUniversalCover}
  \left(\tr_\chi, \id\right)^\ast \iota_E^\ast \cL_E  \cong \iota_E^\ast \cL_E \otimes \epsilon_1^\ast \cL_{E, (\chi, \check{\jmath}\,)},
\end{equation}
where $\epsilon_1 \colon E \times \check{E}_1 \to \check{E}_1$ is the projection onto the second factor.

\subsubsection{} The $\Lambda_S$-linearization of the unipotent bundle $\cU_E$ is somewhat trickier to come by. Rather formally, for $\chi \in \Lambda$, the isomorphism $\cU_E \to \tr_\chi^\ast \cU_E$ is just the push-forward of the isomorphism \eqref{Eq:LinearizationPoincareBundleOnThickening} along the $\Lambda_S$-equivariant morphism $\pi_E$. The key observation (see \Cref{Prop:TensorProductBecomesBaerSum}) is that the isomorphism \eqref{Eq:TranslationPoincareBundleUniversalCover} becomes, after pushing forward along $\pi_E$, an isomorphism of $\cO_E$-modules
\[ \tr_\chi^\ast \cU_E \cong \cU_E +_\rB \epsilon^\ast \chi^\ast \cU_E. \]
(Notice that both the vector bundles $\cU_E$ and $\epsilon^\ast \chi^\ast \cU_E$ are extensions of $\cO_E$ by $\epsilon^\ast \omega_{\check{E}}$; thus their Baer sum is well defined.) 

Accordingly, the tensor product in \eqref{Eq:LinearizationPoincareBundleOnThickening} is now replaced by a `sum'. To be more precise about what this possibly means, recall how the Baer sum of the extensions $(\cU_E)$ and $(\epsilon^\ast \chi^\ast \cU_E)$ is constructed. First, consider the $\cO_E$-submodule 
\[ V \subseteq \cU_E \oplus \epsilon^\ast \chi^\ast \cU_E \]
made of pairs $(v, w)$ whose components have same projection in $\cO_E$. The $\cO_E$-module $V$ is an extension of $\cO_E$ by $\epsilon^\ast (\omega_{\check{E}} \oplus \omega_{\check{E}})$, and the Baer sum in question is its push-out along the sum map $\omega_{\check{E}} \oplus \omega_{\check{E}} \to \omega_{\check{E}}$. For  a pair $(v, w)$, the aforementioned `sum' $v + w$ is the image of $(v, w)$ in $\cU_E +_\rB \epsilon^\ast \chi^\ast \cU_E$. Summing up, we proved the following. 

\begin{proposition} \label{Prop:LinearizationCanonicalExtension}With the notation above, the $\Lambda_S$-linearization of the unipotent bundle $\cU_E$ is given, for $\chi \in \Lambda$, by the isomorphism
\begin{equation} \label{Eq:LinearizationCanonicalExtension}
\cU_E \too \tr_\chi^\ast \cU_E, \quad v \longmapsto v + q(v).\left\langle \chi, \check{\jmath}\, \right\rangle,
\end{equation}
where $q \colon \cU_E \to \cO_E$ is the projection in the datum of the extension $(\cU_E)$.
\end{proposition}

  The meaning of the preceding formula is unveiled by the following.

  \begin{remark} \label{Rmk:InterpretationInfinitesimalSection} The equality $\cU_E = \pi_{E \ast} \iota_E^\ast \cL_E$ implies
    \[ \chi^\ast \cU_E = \check{\epsilon}_{1\ast} \cL_{E, (\chi, \check{\jmath}\,)},\] where $\check{\epsilon}_{1} \colon \check{E}_1 \to S$ is the structural morphism. By means of the preceding, the trivialization $\langle \chi, \check{\jmath}\,\rangle$ of the line bundle $\cL_{E, (\chi, \check{\jmath}\,)}$ can be seen as a section of the vector bundle $\chi^\ast \cU_E$ on $S$.  Moreover, the sum in  expression \eqref{Eq:LinearizationCanonicalExtension} does make sense for the following reason: the projection in $\cO_S = \cL_{E, (\chi, \check{e})}$ of the section $\langle \chi, \check{\jmath}\, \rangle$ of the extension $\chi^\ast \cU_E$ is $\langle \chi, \check{e}\rangle = 1$ (where $\check{e}$ is the neutral section of $\check{E}$);  thus the sections $v$ and $q(v).\langle \chi, \check{\jmath}\, \rangle$ have equal projection in $\cO_S$.
 \end{remark}

\subsection{Universal cover of the universal vector extension} \label{sec:UniversalCoverUniversalVectorExtension}

\begin{definition} The \emph{universal cover} of the universal vector extension $A^\natural$ is 
\[ E^\natural := \bbP(\cU_E) \smallsetminus \bbP(\epsilon^\ast \omega_{\check{E}}).\]
\end{definition}

\begin{remark} When $S$ is a point, the name `universal cover' is well deserved as the topological space $E^\natural$ is contractible (see \Cref{Prop:UniversalCoverContractible}) and comes with a covering map  $E^\natural \to A^\natural$ (see \Cref{Prop:UniversalCoverUniversalVectorExtension}).
\end{remark}

Let $\check{E}_1$ be the first-order thickening of $\check{E}$ along the neutral section and $\check{\jmath} \colon \check{E}_1 \to \check{E}$ the closed immersion. As explained in \Cref{Rmk:InterpretationInfinitesimalSection}, for $\chi \in \Lambda$, the trivialization $\langle \chi, \check{\jmath}\,\rangle$ of the line bundle $\cL_{E, (\chi, \check{\jmath}\,)}$ on $\check{E}_1$ defines a section of the vector bundle $\chi^\ast \cU_E$ on $S$. Moreover, the projection of $\langle \chi, \check{\jmath}\,\rangle$ in $\cO_S$ is $1$; thus the section $\langle \chi, \check{\jmath}\,\rangle$ defines a splitting of the short exact sequence $\chi^\ast (\cU_E)$. Therefore, $\langle \chi, \check{\jmath}\,\rangle$ defines an $S$-valued point $\chi^\natural$ of the universal cover $E^\natural$. Let
\[ i^\natural \colon \Lambda_S \too E^\natural, \quad \chi \longmapsto \chi^\natural := \langle \chi, \check{\jmath}\,\rangle \]
be the so-defined morphism of $S$-analytic spaces. For $\chi \in \Lambda$, the projection of $\chi^\natural$ in $E$ is by definition $\chi$; thus the morphism $i^\natural$ is a closed immersion. The isomorphism of $\cO_E$-modules $u^\ast \cU_A \cong \cU_E$ introduced in \Cref{sec:AlternativeDescriptionCanonicalExtensionUniversalCover} induces an isomorphism of $E$-analytic spaces $A^\natural \times_A E \cong E^\natural$. It follows that $E^\natural$ is naturally endowed with a structure of $S$-analytic space in groups and the natural morphism of $S$-analytic spaces
\[ u^\natural \colon E^\natural \too A^\natural\]
is a group morphism.

\begin{theorem} \label{Prop:UniversalCoverUniversalVectorExtension} The map $i^\natural$ is a group morphism with image the kernel of $u^\natural$; that is, the following sequence of $S$-analytic groups is short exact:
\[ 0 \too \Lambda_S \stackrel{i^\natural}{\too} E^\natural \stackrel{u^\natural}{\too} A^\natural \too 0.\]
\end{theorem}

\begin{proof} The group law on the affine bundle $A^\natural = \bbP(\cU_A) \smallsetminus \bbP(\alpha^\ast \omega_{\check{A}})$ is given by an isomorphism of $\cO_A$-modules
\[ \psi_A \colon  \pr_1^\ast \cU_A +_\rB \pr_2^\ast \cU_A  \stackrel{\lowsim}{\too}  \mu_A^\ast \cU_A,\]
where $\mu_A, \pr_1, \pr_2 \colon A \times_S A \to A$ are, respectively, the group law, the first and the second projection on $A$. According to \Cref{Rmk:GroupStructureViaRigidification}, or better its rigid analytic analogue, such an isomorphism $\psi_A$ is the push-forward along the projection
\[ \left(A \times_S \check{A}_1\right) \times_{\check{A}_1} \left(A \times_S \check{A}_1\right) = A \times_S A \times \check{A}_1 \too A \times_S A \]
 of the isomorphism $(\id_{A} \times \check{\jmath}, \id_{A} \times \check{\jmath})^\ast \phi_A$, where
\[ \phi_A \colon  \pr_1^\ast \cL_A \otimes \pr_2^\ast \cL_A \stackrel{\lowsim}{\too} \mu_A^\ast \cL_A\]
is the rigidification of the Poincar\'e bundle $\cL_A$ and
\[ \left(\id_{A} \times \check{\jmath}, \id_{A} \times \check{\jmath}\,\right) \colon \left(A \times_S \check{A}_1\right) \times_{\check{A}_1} \left(A \times_S \check{A}_1\right) \too \left(A \times_S \check{A}\right) \times_{\check{A}} \left(A \times_S \check{A}\right)
\]
the closed immersion. By design, the group law of $E^\natural$ is constructed from that of $A^\natural$. That is, the group law on the affine bundle $E^\natural =  \bbP(\cU_E) \smallsetminus \bbP(\epsilon^\ast \omega_{\check{E}})$ is defined by the isomorphism of $\cO_E$-modules
\[ \psi_E := (u \times u)^\ast \psi_A \colon  \pr_1^\ast \cU_E +_\rB \pr_2^\ast \cU_E \stackrel{\lowsim}{\too} \mu_E^\ast \cU_E,\]
the isomorphism $u^\ast \cU_A \cong \cU_E$ (and with the obvious notation) being allowed for.
Thanks to \eqref{Eq:LinearizationCanonicalExtension} and with the notation therein, the $\Lambda_S$-linearization of the unipotent bundle $\cU_E$ is the datum of the isomorphisms, for $\chi \in \Lambda$,
\[ \lambda_\chi  \colon v \longmapsto v + q(v) \chi^\natural,\]
where $q \colon \cU_E \to \cO_E$ is the projection in the datum of the extension $(\cU_E)$. 

Given $\chi, \chi' \in \Lambda$, via the isomorphism $\psi_E$, the wanted additivity $\chi^\natural +\chi'^\natural = (\chi + \chi')^\natural$ is seen to be  nothing but the compatibility $\lambda_{\chi + \chi'} = \tr_\chi^\ast \lambda_{\chi'} \circ \lambda_\chi$ to which is subsumed the $\Lambda_S$-linearization of $\cU_E$.

This substantially also proves the remaining assertions. Indeed, the affine bundle $A^\natural$ is the quotient of $E^\natural$ by the natural action of $\Lambda_S$ on it induced by the $\Lambda_S$-linearization of $\cU_E$. Moreover, from this point of view, the map $u^\natural$ is just the quotient morphism, and the above considerations show that the action of $\Lambda_S$ on $E^\natural$ is described, for $\chi \in \Lambda$, as the translation by $\chi^\natural$. This amounts to saying that the sequence of $S$-analytic groups
\[ 0 \too \Lambda_S \stackrel{i^\natural}{\too} E^\natural \stackrel{u^\natural}{\too} A^\natural \too 0\]
is short exact, thus concluding the proof.
\end{proof}

\subsection{A computation on toric bundles} The next task is acquiring a better understanding of the map $i^\natural$. This is done via a computation on toric bundles whose proof is clearer when stated in a broader generality. So let us  momentarily reset 
our notation, and let $X$ be a separated $S$-analytic space, $\Lambda$ a free abelian group of finite rank $n$, $\Lambda \to \Pic(X)$, $\chi \mapsto L_\chi$ a group homomorphism and $P$ the $X$-analytic space whose points $s$ with values on a $X$-analytic space $f \colon X' \to X$ form the set of data, for $\chi \in \Lambda$, of a trivialization $\langle \chi, s \rangle$ of the line bundle $f^\ast L_\chi$. Moreover, the trivializations above satisfy, for $\chi, \chi' \in \Lambda$, the relation\footnote{\label{Footnote:AbuseNotationToricBundle}Several abuses of notation have been perpetrated
here. Rather than isomorphism classes of line bundles, one should fix, for $\chi \in \Lambda$, a line bundle $L_\chi$ and, for $\chi, \chi' \in \Lambda$, isomorphisms $L_\chi \otimes L_{\chi'} \cong L_{\chi + \chi'}$ through which the formula $\langle \chi, s \rangle \otimes \langle \chi, s \rangle = \langle \chi + \chi', s \rangle$ ought to be understood.}
\begin{equation} \label{Eq:RelationTrivializationLineBundles} \langle \chi, s \rangle \otimes \langle \chi, s \rangle = \langle \chi + \chi', s \rangle.
\end{equation}
Let $p \colon P \to X$ be the projection. The split torus $T$ over $S$ with group of characters $\Lambda$ acts naturally on $P$ by the rule defined, for an $S$-analytic space $S'$,  $S'$-valued points $s$ of $P$ and $t$ of $T$ and a character $\chi \in \Lambda$, by 
\[ \langle \chi, ts \rangle = \chi(t) \langle \chi, s \rangle. \]

Let $u \colon S \to P$ be a morphism of $S$-analytic spaces, $P_1$ the first-order thickening of $P$ along the section $u$ and $j \colon P_1 \to P$ the closed immersion. To make the notation more flexible, given a finite morphism of $K$-analytic spaces $Z \to S$ and a coherent $\cO_Z$-module $F$, let us denote again by $F$ its push-forward onto $S$.  Consider the $S$-valued point $x = p(u)$ of $X$, the first-order thickening $X_1$ of $X$ along the section $x$, the closed immersion $i \colon X_1 \to X$ and the following commutative and exact diagram of $\cO_S$-modules:
\begin{equation} \label{Eq:ProofProjectionOntoTrivialPartCanonicalExtension1}
\begin{tikzcd}[column sep=20pt, row sep=15pt]
& 0 \ar[d] & 0 \ar[d] & 0 \ar[d] & \\
0 \ar[r] & x^\ast \Omega^1_{X/S} \otimes x^\ast L_\chi \ar[d, "\rd p \otimes \id"] \ar[r] & i^\ast L_\chi \ar[d] \ar[r]  & x^\ast L_\chi \ar[d, equal] \ar[r] & 0 \\
0 \ar[r] & u^\ast \Omega^1_{P/S} \otimes x^\ast L_\chi \ar[d] \ar[r] & j^\ast p^\ast L_\chi \ar[d] \ar[r] & x^\ast L_\chi \ar[d] \ar[r] & 0 \\
0 \ar[r] & u^\ast \Omega^1_{P/X} \otimes x^\ast L_\chi \ar[d] \ar[r, "\sim"] & j^\ast p^\ast L_\chi / i^\ast L_\chi \ar[d] \ar[r] & 0 \\
&0 & 0\rlap{.}
\end{tikzcd}
\end{equation}

The map $i^\ast L_\chi \to j^\ast p^\ast L_\chi$ is given by adjunction with respect to the morphism $p_1 \colon P_1 \to X_1$ induced by~$p$. The isomorphism $T \times_S P \to  P \times_X P$ given by the action of $T$ yields an isomorphism of $\cO_S$-modules $u^\ast \Omega^1_{P/X} \cong e^\ast \Omega^{1}_{T/S} =:\omega_{T}$, where $e$ is the neutral section of $T$. Consider the homomorphism 
\[ q \colon j^\ast p^\ast L_\chi \too \omega_T   \otimes x^\ast L_\chi \]
 defined via the isomorphisms $u^\ast \Omega^1_{P/X} \otimes x^\ast L_\chi \cong j^\ast p^\ast L_\chi / i^\ast L_\chi$ and $u^\ast \Omega^1_{P/X} \cong \omega_T$.

\begin{proposition} \label{Prop:InfinitesimalSectionToricBundle} With the notation above,
\[ q(\langle \chi, j\rangle) = \chi^\ast \tfrac{\rd z}{z} \otimes u,\]
where $z$ is the coordinate function on $\bbG_{m, S}$ and $\chi$ is seen as a character $T \to \bbG_{m, S}$.  
\end{proposition}

\begin{proof}
The argument is just a tedious \emph{d\'evissage} until reaching the case $\Lambda = \bbZ$, $\chi = 1$, $X = S$, $L_1 = \cO_S$, $u = 1$,  for which the result is substantially trivial.

\medskip

\textit{First step.}~ To begin with, reduce to the case $\Lambda = \bbZ$ and $\chi = 1$. Consider the $X$-analytic space $P_\chi$ parametrizing trivializations of the line bundle $L_\chi$ and the morphism $\pr_\chi \colon P \to P_\chi$ sending a point $s$ of $P$ with values in an $X$-analytic space $X'$ to the trivialization $\langle \chi , s\rangle$. Then, by design, the trivialization $\langle \chi, j\rangle$ can be seen as the composite morphism $\pr_\chi \circ j \colon P_1 \to P_\chi$. The latter factors through the first-order thickening $P_{\chi, 1}$ of $P_\chi$ along the section $\langle \chi, u\rangle = \pr_\chi(u)$, giving rise to the commutative diagram
\[
\begin{tikzcd}
P_1 \ar[r, "j"] \ar[d, "\pr_{\chi, 1}"'] & P \ar[d, "\pr_\chi"] \\
P_{\chi, 1} \ar[r, "j_{\chi}"]& P_\chi\rlap{,}
\end{tikzcd}
\]
where $j_\chi$ is the closed immersion and $\pr_{\chi, 1}$ the factorization of $\pr_\chi$. Upon letting $p_\chi \colon P_\chi \to X$ be the projection, the homomorphism $q_\chi \colon j_\chi^\ast p_\chi^\ast L_\chi \to \omega_{\Gm} \otimes x^\ast L_\chi$, defined analogously to $q$, fits in the  commutative diagram of $\cO_S$-modules
\[ 
\begin{tikzcd}
j_\chi^\ast p_\chi^\ast L_\chi \ar[d] \ar[r, "q_\chi"] & \pr_\chi(u)^\ast \omega_{\Gm} \otimes x^\ast L_\chi \ar[d, "\rd \chi \otimes \id"] \\
j^\ast p^\ast L_\chi  \ar[r, "q"] & \omega_T \otimes x^\ast L_\chi,
\end{tikzcd}
\]
where the leftmost vertical arrow is given by adjunction with respect to the map $\pr_{\chi, 1}$ and the rightmost one by pull-back of differential forms along the character $\chi$. In particular, it suffices to show that the trivialization of the line bundle $L_\chi$ given by the $P_{\chi, 1}$-valued point $j_\chi$ of $P_\chi$ is mapped to $\tfrac{\rd z}{z} \otimes u$ by $q_\chi$.

\medskip

\textit{Second step.} Suppose $\Lambda = \bbZ$ and $\chi = 1$. To simplify notation,  simply write $L$ instead of $L_\chi$. The aim of this second step is to reduce to the case $X = S$, $L = \cO_S$ and $u = 1$. In order to do this, consider the fiber $P_x$ of $P$ at $x$, which can also be seen  as the principal $\Gm$-bundle over $S$ associated with the line bundle $x^\ast L$ on $S$. In this case, the identity map of $S$ plays the role of the section $x$, and the line bundle $j^\ast p^\ast L$ on $P_1$ is replaced by the line bundle $\pi^\ast x^\ast L$ on $P_{x,1}$, where $P_{x, 1}$ is the first-order thickening of $P_x$ at $u$ and $\pi \colon P_{x, 1} \to S$ is the structural morphism. Diagram \eqref{Eq:ProofProjectionOntoTrivialPartCanonicalExtension1} for the principal bundle $P_x$ reads more simply as the following: 
\[
\begin{tikzcd}[column sep=20pt, row sep=15pt]
& 0 \ar[d, equal]& 0 \ar[d] & 0 \ar[d] & \\
0 \ar[r, equal]& 0 \ar[d] \ar[r] & x^\ast L \ar[d] \ar[r, equal] & x^\ast L \ar[d, equal] \ar[r] & 0 \\
0 \ar[r] & x^\ast L \otimes u^\ast \Omega^{1}_{P_x/S} \ar[d, equal] \ar[r] & \pi^\ast x^\ast L \ar[d] \ar[r] & x^\ast L \ar[d] \ar[r] & 0 \\
0 \ar[r] & x^\ast L \otimes u^\ast \Omega^{1}_{P_x/S}\ar[d] \ar[r, "\sim"] & \pi^\ast x^\ast L / x^\ast L \ar[r] \ar[d] & 0 \\
& 0 & 0\rlap{,} 
\end{tikzcd}
\]
where $x^\ast L \to \pi_\ast x^\ast L$ is natural map. The restriction map from $P_1$ to $P_{x, 1}$ furnishes a homomorphism of commutative diagrams of $\cO_S$-modules from \eqref{Eq:ProofProjectionOntoTrivialPartCanonicalExtension1} to the above, matching the entries in the obvious manner. In particular, the projection $q_x \colon \pi^\ast x^\ast L \to x^\ast L \otimes \omega_{\Gm}$ defined similarly to $q$ sits in the following diagram:
\[ 
\begin{tikzcd}
j^\ast p^\ast L \ar[r, "q"] \ar[d] & x^\ast L \otimes \omega_{\Gm}\ \ar[d, equal] \\
\pi^\ast x^\ast L \ar[r, "q_x"] & x^\ast L \otimes \omega_{\Gm}.
\end{tikzcd}
\]

Strictly speaking, the above argument permits one  to reduce to the case $X = S$ and to a line bundle $L$ on~$S$. However, by means of the isomorphism $\cO_S \cong L$ induced by the given trivialization $u$, one is finally led back to the case of the trivial line bundle on $S$ and $u = 1$.

\medskip

\textit{Third step.} Suppose $X= S$, $L = \cO_S$ and $u = 1$. Let $z$ be the coordinate function on $P = \Gm$. The $\cO_S$-algebra $\cO_{P_1}$ is then identified with $\cO_{S}[z] / (z-1)^2$ and its ideal generated by $z-1$ with the $\cO_S$-module $\omega_{\Gm}$, so that
\[ \cO_{S}[z] / (z-1)^2 = \cO_S \oplus \omega_{\Gm}.\]
In these terms, the map $q$ is just the projection onto $\omega_{\Gm}$. Moreover, and tautologically enough, the closed immersion $j \colon P_1 \to P$ corresponds to the invertible function $z$ on $P_1$. Writing $z = 1 + (z - 1)$ shows that the image of $z$ in $\omega_{\Gm}$ coincides with that of $z - 1$, which in turn corresponds to the invariant differential $\tfrac{\rd z}{z}$.
\end{proof}

\subsection{Relation with the universal vector hull of the fundamental group} In this section the map $i^\natural$ is related to the universal vector hull of the group $\Lambda$. To define the latter, recall that the group $\Lambda$ is by definition the lattice of characters of the torus $\check{T}$. Given $\chi \in \Lambda$, let $\chi \colon \check{T} \to \bbG_{m, S}^\an$ again denote the corresponding character. 

\begin{definition} The \emph{universal vector hull} of $\Lambda_S$ is the $S$-analytic group morphism
\[ \theta_\Lambda \colon \check{T} \too \bbV\left(\omega_{\check{T}}\right), \quad \chi \longmapsto  \chi^\ast \tfrac{\rd z}{z}, \]
where $z$ is the coordinate function on $\Gm$ and $\tfrac{\rd z}{z}$ the corresponding invariant differential form.
\end{definition}

\begin{lemma} \label{Lemma:BasisOfCharactersGivesBasisOfInvariantDifferential} For a basis $\chi_1, \dots, \chi_n$ of $\Lambda$, where $n = \rk \Lambda$, the sections $\theta_\Lambda(\chi_i)$ form a basis of the vector bundle $\omega_{\check{T}}$ on $S$.
\end{lemma}

The map $\theta_\Lambda$ deserves the name `universal vector hull' because of the following property: given a vector bundle $E$ over $S$ and a morphism of $S$-analytic groups $f \colon \Lambda_S \to \bbV(E)$, there exists a unique homomorphism of $\cO_S$-modules $\phi \colon \omega_{\check{T}} \to E$ such that $f = \phi \circ \theta_\Lambda$; see \cite[Example~1.3 a) and Proposition 1.4]{MazurMessing}.
Now, the canonical extension $(\cU_E)$ is by definition the push-out along the injective map $\rd \check{p} \colon \omega_{\check{B}} \to \omega_{\check{E}}$ of the extension $p^\ast (\cU_B)$. The quotient $\cU_E / p^\ast \cU_B$ therefore sits in the following commutative and exact diagram of $\cO_E$-modules:
\begin{equation} \label{Eq:ProjectionOntoTrivialPartCanonicalExtension}
\begin{tikzcd}[column sep=20pt, row sep=15pt]
& 0 \ar[d] & 0 \ar[d] & 0 \ar[d] & \\
0 \ar[r] & \epsilon^\ast \omega_{\check{B}} \ar[d] \ar[r] & p^\ast \cU_B\ar[d]  \ar[r] & \cO_E \ar[d, equal] \ar[r] & 0 \\
0 \ar[r] & \epsilon^\ast \omega_{\check{E}} \ar[d] \ar[r] & \cU_E \ar[d] \ar[r] & \cO_E \ar[d] \ar[r] & 0 \\
0 \ar[r] & \epsilon^\ast \omega_{\check{T}} \ar[d] \ar[r, "\sim"] & \cU_E/p^\ast \cU_{B} \ar[d] \ar[r] & 0 \\
& 0 & 0\rlap{.}
\end{tikzcd}
\end{equation}
The isomorphism $\epsilon^\ast \omega_{\check{T}} \cong \cU_E / p^\ast \cU_B$ allows one to define a projection $q \colon \cU_E \to \epsilon^\ast \omega_{\check{T}}$. In particular, for $\chi \in \Lambda$, this defines a homomorphism of $\cO_S$-modules $q \colon \chi^\ast \cU_E \to \omega_{\check{T}}$. Recall that $\chi^\natural = \langle \chi, \check{\jmath}\, \rangle$ can be seen as a section of $\chi^\ast \cU_E$ (see \Cref{Rmk:InterpretationInfinitesimalSection}).

\begin{proposition} \label{Prop:QuotientLinearizationUniversalExtension} For $\chi \in \Lambda$ we have $q(\langle \chi, \check{\jmath}\, \rangle) = \theta_{\Lambda}(\chi)$.
\end{proposition}

\begin{proof} Up to transliteration, the statement is a special case of \Cref{Prop:InfinitesimalSectionToricBundle}. Needless to say, 
\begin{itemize}
\item the $S$-analytic space therein $X$ plays the role of $\check{B}$, 
\item  the principal bundle $P$ that of $\check{E}$, 
\item the line bundle $L_\chi$ that of $\cL_{B, (c(\chi), \id)}$ and 
\item the section $u$ that of $\langle \chi, \check{e} \rangle$, where $\check{e}$ is the neutral section of $\check{E}$. 
\end{itemize}
For $\chi \in \Lambda$ the section $\langle \chi, j \rangle$ corresponds via this dictionary to the trivialization $\langle \chi, \check{\jmath}\, \rangle$. Nonetheless, the reader might still be lost in translation while trying to see why diagram \eqref{Eq:ProofProjectionOntoTrivialPartCanonicalExtension1} reads as \eqref{Eq:ProjectionOntoTrivialPartCanonicalExtension}. To remedy that, first, notice that in the current framework, the line bundle $\check{e}^\ast \cL_{B, (c(\chi), \id)}$ is always understood to be trivialized via $\langle \chi, \check{e}\rangle$. This should elucidate the omnipresence of the line bundle  $x^\ast L_\chi$ as opposed to the absence of the corresponding line bundle $\check{e}^\ast \cL_{B, (c(\chi), \id)}$. Second, the statement of \Cref{Prop:QuotientLinearizationUniversalExtension} revolves around the vector bundles $\cU_E$ and $p^\ast \cU_B$ on $E$ (rather, their fibers at $\chi$), whereas in \Cref{Prop:InfinitesimalSectionToricBundle} the line bundles $j^\ast p^\ast L_\chi$ and $i^\ast L_\chi$ are considered (better, their push-forward onto $S$). However, the line bundle $j^\ast p^\ast L_\chi$ translates to $\cL_{E, (\chi, \check{\jmath}\,)}$ and, as already observed in
\Cref{Rmk:InterpretationInfinitesimalSection}, the push-forward of $\cL_{E, (\chi, \check{\jmath}\,)}$ onto $S$ coincides with $\chi^\ast \cU_E$. Along a similar line, the vector bundle $i^\ast L_\chi$ on $S$ plays the role of $\chi^\ast p^\ast \cU_B$. Third, in diagram \eqref{Eq:ProjectionOntoTrivialPartCanonicalExtension}, there is nothing whatsoever like $u^\ast \Omega^1_{P/X}$. This is because the projection  $q$ in the statement of \Cref{Prop:InfinitesimalSectionToricBundle} is constructed by further taking into account the isomorphism $\omega_{T} \cong u^\ast \Omega^1_{P/X}$  already implied here.
\end{proof}

Diagram \eqref{Eq:ProjectionOntoTrivialPartCanonicalExtension} permits one to define a morphism of $S$-analytic groups
\[
\pr_u \colon E^\natural \too \bbV\left(\omega_{\check{T}}\right).
\]
Unwinding the definitions, \Cref{Prop:QuotientLinearizationUniversalExtension} is rephrased as follows. 

\begin{theorem} \label{Thm:LinearizationAndUniversalVectorHull} For $\chi \in \Lambda$ we have $\pr_u(\chi^\natural) = \theta_{\Lambda}(\chi)$.
\end{theorem}

\begin{example} Theorem~\ref{Thm:LinearizationAndUniversalVectorHull} is more eloquent when $A$ is an abeloid variety over $K$ with totally degenerate reduction, that is, $A = T / \Lambda$. If $\check{T}$ is the torus with group of characters $\Lambda$ and $\theta_\Lambda \colon \Lambda \to \bbV(\omega_{\check{T}})$ the universal vector hull of $\Lambda$, then
\[ A^\natural = (T \times \bbV(\omega_{\check{T}})) /
\{ (\chi, \theta_\Lambda(\chi)) : \chi \in \Lambda\}.\]
\end{example}

\subsection{Universal cover of affine bundles} \label{sec:UniversalCoverAffineBundles}

\subsubsection{Extensions} Let  $F$ be a vector bundle over $S$, $\phi_A \colon \omega_{\check{A}} \to F$ a homomorphism of $\cO_S$-modules and $(\cF_A)$ the short exact sequence of $\cO_S$-modules obtained by push-out of the canonical extension $(\cU_A)$ along $\phi_A$.  Via the isomorphism $\cU_E \cong u^\ast \cU_A$ obtained in \Cref{sec:AlternativeDescriptionCanonicalExtensionUniversalCover}, the short exact sequence of $\cO_E$-modules $(\cF_E) := u^\ast (\cF_A)$ is seen to be the push-out of $(\cU_E)$ along the homomorphism of $\cO_S$-modules
\[ 
\begin{tikzcd}
\phi_E \colon \omega_{\check{E}} \ar[r, "{(\rd \check{u})^{-1}}"]&\omega_{\check{A}} \ar[r, "\phi_A"] & F,
\end{tikzcd}
\]
where $\rd \check{u} \colon \omega_{\check{A}} \to \omega_{\check{E}}$ is the isomorphism given by pull-back of differential forms along the \'etale morphism~$\check{u}$. On the other hand, by its very definition, the short exact sequence $(\cU_E)$ is itself the push-out of $p^\ast (\cU_B)$ along the $K$-linear homomorphism $\rd \check{p} \colon \omega_{\check{B}} \to \omega_{\check{E}}$. Thus, the preceding considerations furnish an isomorphism $(\cF_E) \cong p^\ast (\cF_B)$ of short exact sequences of $\cO_{E}$-modules, where $(\cF_B)$ is the push-out of the canonical extension $(\cU_B)$ along the homomorphism of $\cO_S$-modules
\[ 
\begin{tikzcd}
\phi_B \colon \omega_{\check{B}} \ar[r, "\rd \check{p}"]& \omega_{\check{E}}  \ar[r, "\phi_E"] & F.
\end{tikzcd}
\]

\subsubsection{Affine bundles} Consider the affine bundle $\pi_A \colon \bbV(\cF_A) \to A$ and the morphism of $S$-analytic spaces
\[ \Phi_A \colon A^\natural \too \bbA(\cF_A)\]
induced by the homomorphism $(\cU_A) \to (\cF_A)$ of short exact sequences of $\cO_A$-modules given by the definition of $(\cF_A)$ as a push-out. The affine bundle $\bbA(\cF_A)$ carries a unique $S$-analytic group structure for which the morphism $\Phi_A$ is a group morphism. The pull-back $\bbV(\cF_A) \times_A E$ to $E$ is by definition the affine bundle $\pi_E \colon \bbA(\cF_E) \to E$ associated with the extension $(\cF_E)$. By transport of structure, the $E$-analytic space $\bbA(\cF_E)$ is an $S$-analytic group such that the natural morphism 
\[ \Phi_E \colon E^\natural \too \bbA(\cF_E)\]
deduced from $\Phi_A$ is a group morphism. The natural action of $\Lambda_S$ on $\bbV(\cF_E)$ given by the natural $\Lambda_S$-linearization of $\cF_E$ is described, for $\chi \in \Lambda$, as the translation by the point $\Phi_E(\chi^\natural)$, 
where $\chi^\natural$ is the $S$-point of the universal vector extension $E^\natural$ considered in \Cref{sec:UniversalCoverUniversalVectorExtension}.

Now, the isomorphism of short exact sequences of $\cO_E$-modules  $(\cF_E) \cong p^\ast (\cF_B)$ permits one to identify the affine bundle $\bbA(\cF_E)$ with the fibered product $\bbA(\cF_B) \times_B E$, where $\pi_B \colon \bbA(\cF_B) \to B$ is the affine bundle associated with $(\cF_B)$. This identification respects the natural $S$-analytic group structures involved and will be implied in what follows. Let $q \colon \bbA(\cF_E) \to \bbA(\cF_B)$ be the projection, so that the following square of $K$-analytic space is Cartesian:
\[
\begin{tikzcd}
\bbA(\cF_E) \ar[d, "\pi_E"] \ar[r, "q"] & \bbA(\cF_B) \ar[d, "\pi_B"] \\
E \ar[r, "p"] & B\rlap{.}
\end{tikzcd}
\]

Let $C_B$ be the cokernel of $\phi_B$. Arguing as for the map $\pr_u \colon E^\natural \to \bbV(\omega_{\check{T}})$
 permits one to define a morphism of $S$-analytic groups $\pr_{u, B} \colon \bbA(\cF_B) \to \bbV(C_B)$. Let $\phi_T \colon \omega_{\check{T}} \to C_B$ the unique map fitting in the following commutative and exact diagram of homomorphism of $\cO_S$-modules:
\begin{equation} \label{Eq:DecompositionUniversalCoverLinearMapMakingPushOut}
\begin{tikzcd}
0 \ar[r] & \omega_{\check{B}} \ar[r, "\rd \check{p}"] \ar[d, "\phi_B"]& \omega_{\check{E}} \ar[r] \ar[d, "\phi_E"] & \omega_{\check{T}} \ar[d, "\phi_T"] \ar[r]& 0\ \\
0 \ar[r] & \im \phi_B \ar[r] & F \ar[r] & C_B \ar[r] & 0.
\end{tikzcd}
\end{equation}
The construction of $(\cF_E)$ as a push-out of $(\cU_E)$ along $\phi_E$ implies that the diagram of $S$-analytic spaces
\[ 
\begin{tikzcd}
E^\natural \ar[r] \ar[d, "\Phi_E"] & \bbV\left(\omega_{\check{T}}\right) \ar[d, "\pr_{u, E}"] \\
\bbA\left(\cF_E\right) \ar[r, "\phi_T"] & \bbV\left(C_B\right)
\end{tikzcd}
\]
is commutative, where the upper horizontal arrow is the projection considered in \Cref{Thm:LinearizationAndUniversalVectorHull} and $\pr_{u, E}:= \pr_{u, B} \circ q  \colon \bbA(\cF_E) \to \bbV(C_B)$. These considerations together with \Cref{Thm:LinearizationAndUniversalVectorHull} prove the following. 

\begin{corollary} With the notation above, for $\chi \in \Lambda$ we have
\[
 \pr_{u, E}(\Phi_E(\chi^\natural)) = \phi_T(\theta_\Lambda(\chi)).
\]
\end{corollary}

\subsubsection{Contractibility of the universal cover} In this final section the contractibility of the space $\bbA(\cF_E)$ above is addressed when $S$ is a $K$-rational point.

\begin{lemma} \label{Lemma:ContractibleOpenInSmooth} Let $\cX$ be a smooth connected admissible formal $R$-scheme with Raynaud's generic fiber $X := \cX_\eta$. Then, for any closed analytic subspace of $Z \subsetneq X$, the open subset $X \smallsetminus Z$ is contractible.
\end{lemma}

\begin{proof} Let $h \colon X \times [0, 1] \to X$ be the deformation retraction  onto the skeleton $\Sk(\cX)$ of $\cX$ given by \cite[Theorem 5.2]{BerkovichContractible}. Since the formal scheme $\cX$ is smooth connected, the skeleton $\Sk(\cX)$ is a singleton, namely the unique preimage of the generic point of $\cX$ under the reduction map $X \to \cX$. According to item~(v) of \emph{loc.~cit.}, for $0 < t \le 1$ and $x \in X$, the local ring at the point $h(x, t)$ is a field. Thus the only closed analytic subspace of $X$ containing $h(x, t)$ is $X$ itself. In particular, the point $h(x, t)$ belongs to $X \smallsetminus Z$, and the statement follows.
\end{proof}

\begin{proposition} \label{Prop:UniversalCoverContractible} The topological space $\bbA(\cF_E)$ is contractible and is a universal cover of $\bbA(\cF_A)$.
\end{proposition}

\begin{proof} Let $F_0$ be the image of the $R$-module $\omega_{\check{\cB}}$ via the map $\phi_B \colon \omega_{\check{B}} \to F$ and $\cF_0$ the push-out of the canonical extension $(\cU_\cB)$ on $\cB$ along $\omega_{\check{\cB}} \to F_0$. For $\check{\chi} \in \check{\Lambda}$, the line bundle $\cL_{B, \check{\chi}} = \cL_{B \rvert B \times \{ \check{c}(\check{\chi})\}}$ extends to a line bundle $\cL_{\cB, \check{\chi}}$ on $\cB$. Consider a basis $\check{\chi}_1, \dots, \check{\chi}_n$ of $\Lambda$ and the smooth connected formal $R$-scheme
  \[
  \cX := \bbP\left(\cF_0\right) \times_\cB \bbP\left(\cO_B \oplus \cL_{\check{\chi}_1}\right) \times_\cB \cdots \times_\cB \bbP\left(\cO_B \oplus \cL_{\check{\chi}_n}\right).
  \]
Now $\bbA(\cF_E)$ is the complement of a Cartier divisor in $X := \cX_\eta$; thus by \Cref{Lemma:ContractibleOpenInSmooth} it is contractible. Since $\bbA(\cF_A)$ is the quotient of $\bbA(\cF_E)$ by the (free) action of $\Lambda$, the statement follows.
\end{proof}

Applying this with $F = \omega_{\check{A}}$ and $\phi_A = \id$ gives that $E^\natural$ is contractible and $E^\natural \to A^\natural$ is a universal cover, which justifies the name universal cover for $E^\natural$.

\renewcommand\thesection{\Alph{section}}
\setcounter{section}{0}

%\appendix
\section*{Appendix. Connections} \label{appendix}
\addcontentsline{toc}{section}{Appendix. Connections}
\refstepcounter{section}

\subsection{Vector bundles on first-order thickenings} \label{sec:VectorBundleFirstOrderThick}

Let $X_0$ and $X_1$ be schemes endowed with a closed immersion $s \colon X_0 \to X_1$ and a morphism $f \colon X_1 \to X_0$ such that $f \circ s = \id_{X_0}$. Suppose that the sheaf of ideals $I := \Ker (\cO_{X_1} \to \cO_{X_0})$ is of square zero. For a quasi-coherent $\cO_{X_1}$-module $F$, consider the short exact sequence
\begin{equation} \tag{$F$} 0 \too IF \too F \too  F /I F \too 0. \end{equation}
The sequence of $\cO_{X_0}$-modules
\begin{equation} \tag*{$f_\ast (F)$} 0 \too f_\ast (IF) \too f_\ast F \too s^\ast F \too 0  \end{equation}
obtained by pushing forward $(F)$ along $f$ is short exact because the morphism $f$ is affine (affineness only depends on the underlying reduced structure). Pushing forward a homomorphism of $\cO_{X_1}$-modules $\phi \colon F \to F'$ yields a homomorphism of short exact sequences of $\cO_{X_0}$-modules $f_\ast \phi \colon f_\ast (F) \to f_\ast (F')$. The so-defined functor
\[ 
\left\{ \cO_{X_1}\textup{-modules}\right\} \too
{\left\{ 
\begin{array}{c}
\textup{short exact sequences} \\
\textup{of $\cO_{X_0}$-modules}
\end{array} \right\}}, \quad F \longmapsto f_\ast (F)
\]
is faithful. Moreover, an isomorphism of $\cO_{X_1}$-modules $f^\ast s^\ast F \to F$ induces a splitting of the short exact $f_\ast (F)$.

\begin{proposition} \label{Prop:IsomorphismAndSplittings} The bijection $ \Hom(f^\ast s^\ast F, F) \to \Hom(s^\ast F, f_\ast F)$ given by adjunction for a vector bundle $F$ on $X_1$ induces a bijection
\[
 \left\{
\begin{array}{c}
\textup{isomorphisms $\rho \colon f^\ast s^\ast F \to F$} \\
\textup{such that $s^\ast \rho = \id_{s^\ast F}$} 
\end{array} \right\}
\cong
 \left\{
\begin{array}{c}
\textup{splittings of the short} \\
\textup{exact sequence $f_\ast (F)$} 
\end{array} \right\}.
\]
\end{proposition}

\begin{proof}  The only thing to show is that, for a splitting $\phi \colon s^\ast F \to f_\ast F$ of the short exact sequence $f_\ast (F)$, the homomorphism $\Phi \colon f^\ast s^\ast F \to F$ obtained by extending $\phi$ $\cO_{X_1}$-linearly  is an isomorphism. To check this, one may reason locally on $X_0$ and choose a splitting $ f_\ast F \cong s^\ast F \oplus s^\ast F \otimes f_\ast I$ of the short exact sequence $f_\ast (F)$. This allows for the identities $I F = I \otimes F$ and $f_\ast (IF) = f_\ast I \otimes s^\ast F$, which hold, respectively, because  $F$ is flat and because the ideal $I$ is of square zero. Via these identifications, the splitting $\phi$ is of the form $v \mapsto (v, \epsilon(v))$ for a homomorphism of $\cO_X$-modules $\epsilon \colon s^\ast F \to s^\ast F \otimes f_\ast I$. Write a section of $f^\ast s^\ast F$ as $(v , v')$ for sections $v$ of $s^\ast F$ and $v'$ of $s^\ast F \otimes f_\ast I$; then the map $\Phi$ is defined as $(v, v') \mapsto (v, \epsilon(v) + v')$, where the term $\epsilon(v')$ vanished because the ideal $I$ is of square zero. Such an expression  clearly defines an isomorphism, which concludes the proof.
\end{proof}

\subsection{Tensor product and Baer sums} For simplicity, assume $X_1$ to be the first-order thickening of $\bbV(E^\vee)$ along its zero section $s$, where $E$ is a vector bundle on $X_0$ and $\bbV(E^\vee)$ the total space of its dual. Concretely, the scheme $X_1$ is the spectrum of the $\cO_{X_0}$-module $\cO_{X_0} \oplus E$ endowed with an $\cO_{X_0}$-algebra structure defined by the formula $(a, v)\cdot(b,w) = (ab, aw+bv)$.
In what follows, it will be important to distinguish whether tensor products are taken with respect to $\cO_{X_0}$ or $\cO_{X_1}$. To mark the difference but at the same time make the notation lighter, for $i = 0, 1$ and $\cO_{X_i}$-modules $V$ and $V'$, write $V \otimes_i V'$ instead of $V \otimes_{\cO_{X_i}} V'$. For a vector bundle $V$ on $X_1$, set $V_0 := s^\ast V$, so that the vector bundle  $f_\ast V$ on $X_0$ is an extension of $V_0$ by $V_0 \otimes_0 E$.  In order to make sense of the following statement, observe that, for vector bundles $V$ and $V'$ on $X_1$, the $\cO_{X_0}$-modules $f_\ast V \otimes_0 V_0'$, $f_\ast V' \otimes_0 V_0$ and $f_\ast (V \otimes_{1} V')$ are all extensions of $V_0 \otimes_0 V'_0$ by $E \otimes_0 V_0 \otimes_0 V'_0$.

\begin{proposition} \label{Prop:TensorProductBecomesBaerSum} The extension $f_\ast (V \otimes_{1} V')$ is the Baer sum of $f_\ast V \otimes_0 V_0'$ and $f_\ast V' \otimes_0 V_0$.
\end{proposition}

\begin{proof} Unfortunately, the argument is quite clumsy and goes through the explicit construction of the Baer sum in question. To recall it,  let $p \colon f_\ast V \to V_0$ and $p' \colon f_\ast V' \to V_0'$ denote the homomorphisms given by restriction to $X_0$, and consider the $\cO_{X_0}$-submodule $ W\subseteq   (V_0  \otimes_0 f_\ast V' ) \oplus (f_\ast V \otimes_0 V_0')$ made of pairs whose components have the same image in $V_0 \otimes_0 V'_0$ via, respectively,  $p \otimes \id_{V'_0}$ and $\id_{V_0} {\otimes} p'$. The vector bundle $W$ on $X_0$ fits into the following short exact sequence of $\cO_{X_0}$-modules:
\begin{equation} \tag{$W$}
0 \too (E \otimes_0 V_0 \otimes_0 V'_0)^{\oplus 2} \too W \too V_0 \otimes_0 V'_0 \too 0. 
\end{equation}
The Baer sum mentioned in the statement is the push-out of the short exact sequence $(W)$ along the sum map $(E \otimes_0 V_0 \otimes_0 V'_0)^{\oplus 2} \to E \otimes_0 V_0 \otimes_0 V'_0$. This being said, consider the natural epimorphism of $\cO_{X_0}$-modules $\phi \colon f_\ast V \otimes_{0} f_\ast V' \to f_\ast (V \otimes_{1} V')$
given by universal property of tensor product. The vector bundle $E$, seen as a sheaf of ideals of $f_\ast \cO_{X_1}$, is of square zero; hence $\Ker \phi = (E \otimes_0 V_0) \otimes_0 (E \otimes_0 V'_0)$. The right-hand side of the previous equality is seen to also be  the intersection of the kernels of the homomorphisms $p \otimes \id_{f_\ast V'}$ and $\id_{f_\ast V'}  {\otimes} p'$. That is, upon setting 
\[ \psi = (p \otimes \id_{f_\ast V'}, \id_{f_\ast V'}  {\otimes} p') \colon f_\ast V \otimes_{0} f_\ast V'  \too (V_0 \otimes_0 f_\ast V') \oplus ( f_\ast V \otimes_0 V_0'), \]
the homomorphisms $\phi$ and $\psi$ share the same kernel. Moreover, the identity 
\[ (p \otimes \id_{V_0'}) \circ (\id_{f_\ast V} {\otimes} p') = p \otimes p' =  (\id_{V_0}  \otimes p') \circ ( p {\otimes}  \id_{f_\ast V'})  \]
of maps $f_\ast V \otimes_0 f_\ast V' \to V_0 \otimes_0 V_0'$ implies that the image of $\psi$ is $W$. Consequently, the homomorphism $\phi$ factors uniquely through an epimorphism $\tilde{\phi} \colon W \to f_\ast (V \otimes_1 V') $ of $\cO_{X_0}$-modules. The kernel of $\phi$ is contained in that of $p \otimes p'$, and the homomorphism $\tilde{\phi}$ acts on the quotient
\[ 
\frac{\ker (p \otimes p')}{\ker \phi} = \frac{K + K'}{K \cap K'} \cong \frac{K}{K \cap K'} \oplus \frac{K'}{K \cap K'}= (E \otimes_0 V_0 \otimes_0 V_0')^{\oplus 2}
\]
as the sum map, where $K$ and $K'$ are  the kernels of the homomorphisms $p \otimes \id_{f_\ast V'}$ and $\id_{f_\ast V} \otimes p'$, respectively. In other words, the homomorphism $\tilde{\phi}$ fits into the following commutative and exact diagram of $\cO_{X_0}$-modules:
\[ 
\begin{tikzcd}
0 \ar[r] & (E \otimes_0 V_0 \otimes_0 V'_0)^{\oplus 2} \ar[d, "\textup{sum}"] \ar[r]  & W \ar[d, "\tilde{\phi}"]  \ar[r] & V_0 \otimes_0 V'_0 \ar[d, equal] \ar[r] & 0\ \\
0 \ar[r] & E \otimes_0 V_0 \otimes_0 V'_0 \ar[r] & f_\ast (V \otimes_{1} V') \ar[r] & V_0 \otimes_0 V'_0 \ar[r] & 0,
\end{tikzcd}
\]
where the upper row is the short exact sequence $(W)$.  To put it differently, the lower row is the push-out of the upper one along the sum map; that is, the extension $f_\ast (V \otimes_1 V')$ is the desired Baer sum.
\end{proof}

\subsection{Connections} Let $S$ be a scheme, $f \colon X \to S$ a separated morphism of schemes, $\Delta_X \colon X \to X \times_S X$ the diagonal morphism and $I := \Ker (\cO_{X \times_S X} \to \Delta_{X \ast} \cO_X)$ its augmentation ideal. With this notation, the sheaf of differentials (denoted by $\Omega^1_{f}$ or $\Omega^1_{X/S}$) relative to $f$ is the $\cO_X$-module $\Delta_X^\ast I$. Let $\Delta_{X, 1}$ be the first infinitesimal neighbourhood of the diagonal, that is, the closed subscheme of $X \times_S X$ defined by the sheaf of ideals $I^2$. For $i = 1, 2$, let $p_i \colon \Delta_{X, 1} \to X$ be the morphism induced by the $\supth{i}$ projection.  The $\cO_X$-module $\cJ_f^1 := p_{1 \ast} \cO_{\Delta_{X, 1}}$ is called the \emph{sheaf of first-order jets}. Each $p_i$ induces a homomorphism of $f^{-1} \cO_S$-algebras $j_i \colon \cO_X \to \cJ^1_f$. The homomorphism of $f^{-1}\cO_S$-modules $\rd_f := j_2 - j_1 \colon \cO_{X} \to \Omega^1_f $ is called the \emph{canonical derivation}.

\begin{definition}   A \emph{connection} on a vector bundle $F$ on $X$ is an isomorphism $\nabla \colon p_1^\ast F \to p_2^\ast F$ of vector bundles over $\Delta_{X, 1}$ whose restriction to the diagonal $\Delta_X^\ast \nabla$ is the identity of $F$.
\end{definition}

\subsection{Atiyah extension} \label{sec:AtiyahExtensions} The kernel of the restriction map $\cO_{\Delta_{X, 1}}  \to i_\ast \cO_X$, where $i \colon X \to \Delta_{X, 1}$ is the closed immersion induced by the diagonal, is by definition of square zero. This allows one to adopt the notation introduced in \Cref{sec:VectorBundleFirstOrderThick} with $X_0 = X$, $X_1 = \Delta_{X, 1}$, $s = i$ and $f = p_1$. Pushing forward the $\cO_{\Delta_{X, 1}}$-module $p_2^\ast F$ along $p_1$ yields the short exact sequence
\begin{equation}
\tag{$\cJ^1_{f}(F)$}
0 \too \Omega^1_{f} \otimes F\too \cJ^1_{f}(F) \too F \too 0 
\end{equation}
of $\cO_S$-modules, where  $\cJ^1_{f}(F) := p_{1 \ast} p_2^\ast F$ is the \emph{$\cO_{X}$-module of first-order jets of\, $F$}. Applied to the $\cO_{\Delta_{X,1}}$-module $p_2^\ast F$, \Cref{Prop:IsomorphismAndSplittings} implies the following. 

\begin{proposition} \label{Prop:CharacterizationConnections}The bijection $\Hom(p_1^\ast F, p_2^\ast F) \to \Hom(F, \cJ^1_{X/S}(F))$ given by adjunction induces a bijection
\[
\left\{
\begin{array}{c}
\textup{connections on $F$}
\end{array} \right\}
\cong
 \left\{
\begin{array}{c}
\textup{splittings of the short} \\
\textup{exact sequence $(\cJ^1_{X/S}(F))$} 
\end{array} \right\}.
\]
\end{proposition}

Set $\At_{f}(F) := \cHom(E, \cJ^1_{f}(F))$. The short exact sequence of $\cO_X$-modules
\begin{equation} \tag{$\At_{f}(F)$} 0 \too \Omega^1_{f} \otimes \cEnd F \too \At_{f} (F) \too \cEnd F \too 0
\end{equation}
obtained as the tensor product of $(\cJ^1_{f}(F))$ with $F^\vee$ is called the \emph{relative Atiyah extension of\, $F$}. When $F$ is a line bundle, $\cEnd F \cong \cO_X$, and splittings of the Atiyah extension $(\At_f(F))$ are in bijection with those of $(\cJ^1_f(F))$, thus with connections on $F$.

\subsection{Infinitesimal rigidifications} \label{sec:InfinitesimalRigidifications}  Let $\pr_1, \pr_2 \colon X \times_{S} X \to X$ be, respectively, the first and the second projection and $x \colon S \to X$ a section of the structural morphism $f \colon X \to S$. Let $X_1$ be the first-order thickening of $X$ along $x$, $\iota \colon X_1 \to X$ the closed immersion and $\pi = f \circ \iota \colon X_1 \to S$ the structural morphism.

\begin{definition} An \emph{infinitesimal rigidification of\, $F$ at $x$}  is an isomorphism of $\cO_{X_1}$-modules $\rho \colon \pi^\ast x^\ast F \to \iota^\ast F$ such that $x^\ast \rho$ is the identity.
\end{definition}

Note that if $\nabla \colon p_1^\ast F \to p_2^\ast F$ is a connection on $F$, then the homomorphism of $\cO_{X_1}$-modules $\tau^\ast \nabla$ is an infinitesimal rigidification of $F$ at $x$. Now, it is possible to give a characterization of infinitesimal rigidifications similar to that of connections. For, the augmentation ideal of the closed immersion $x_1 \colon S \to X_1$ induced by $x$ is of square zero. This permits one to employ the conventions introduced in \Cref{sec:VectorBundleFirstOrderThick} with $X_0 = S$, $s = x_1$ and $f = \pi$. For instance, if $F$ is a vector bundle, then pushing forward the $\cO_{X_1}$-module $\iota^\ast F$ along $\pi$ yields the short exact sequence
\begin{equation}
\tag*{$\pi_\ast (\iota^\ast F)$}
0 \too x^\ast \Omega^1_{X/S} \otimes x^\ast F\too \pi_\ast \iota^\ast F \too x^\ast F \too 0 
\end{equation}
of $\cO_S$-modules. Applied to $\cO_{X_1}$-module $\iota^\ast F$, \Cref{Prop:IsomorphismAndSplittings} reads as follows. 

\begin{proposition} \label{Prop:RigidificationsAndSplittings} The bijection $\Hom(\pi^\ast x^\ast  F, \iota^\ast F) \to \Hom(x^\ast F, \pi_\ast \iota^\ast F)$ given by adjunction induces a bijection
\[
\left\{
\begin{array}{c}
\textup{infinitesimal} \\
\textup{rigidifications of $F$ at $x$} 
\end{array} \right\}
\cong
 \left\{
\begin{array}{c}
\textup{splittings of the short} \\
\textup{exact sequence $\pi_\ast (\iota^\ast F)$} 
\end{array} \right\}.
\]
\end{proposition}

Let $\tau \colon X_1 \to \Delta_{X, 1}$ be the morphism determined by $p_{1} \circ \tau = x \circ \pi$ and $p_{2} \circ \tau = \iota$. Consider the homomorphism $\psi \colon p_2^\ast F \to \tau_\ast \iota^\ast F$ of $\cO_{\Delta_{X, 1}}$-modules given by adjunction (note the equality $\iota^\ast F = \tau^\ast p_{2}^\ast F$). Pushing it forward along the morphism $p_1$ gives a homomorphism 
\[ \pi_\ast \psi \colon (\cJ^1_{X/S}(F)) \too p_{1 \ast} \tau_\ast (\iota^\ast F)\] of short exact sequences of $\cO_X$-modules. Now, by the definition of $\tau$, the following square is commutative:
\begin{center}
\begin{tikzcd}
X_1 \ar[r, "\pi"] \ar[d, "\tau"] & S\ \ar[d, "x"] \\
\Delta_{X, 1} \ar[r, "p_1"] & X\rlap{.}
\end{tikzcd}
\end{center}
Therefore, the short exact sequence $p_{1 \ast} \tau_\ast (\iota^\ast F)$ of $\cO_X$-modules is nothing but the push-forward along the closed immersion $x$ of the short exact sequence $\pi_\ast(\iota^\ast F)$ of $\cO_S$-modules. 

\begin{proposition} \label{Prop:RigidificationsAndPrincipalParts}  The homomorphism $ \phi \colon x^\ast (\cJ^1_{X/S}(F)) \to \pi_\ast(\iota^\ast F)$ of short exact sequence of\, $\cO_S$-modules adjoint to $\pi_\ast \psi$ is an isomorphism.
\end{proposition}

\begin{proof} The homomorphism $\psi \colon p_2^\ast F \to \tau_\ast \iota^\ast F$ restricted to the diagonal is the evaluation at $x$. Therefore, the homomorphism $\pi_\ast \psi$ of short exact sequence of $\cO_X$-modules is the following commutative diagram: 
\[ 
\begin{tikzcd}[row sep=15pt]
0 \ar[r] & \Omega^1_{X / S} \otimes F \ar[r] \ar[d, "\ev_x"] & \cJ^1_{X/ S}(F) \ar[r] \ar[d, "\pi_\ast \psi"]& F \ar[r] \ar[d, "\ev_x"] & 0 \\
0 \ar[r] & x_\ast x^\ast \Omega^1_{X / S} \otimes x_\ast x^\ast F \ar[r] & x_\ast \pi_\ast \iota^\ast F \ar[r] & x_\ast x^\ast F \ar[r] & 0\rlap{,}
\end{tikzcd}
\]
where $\ev_x$ is the evaluation at $x$. The homomorphism $\phi$, which is the adjoint to $\pi_\ast \psi$, reads as the  commutative diagram
\[ 
\begin{tikzcd}[row sep=15pt]
0 \ar[r] &  x^\ast \Omega^1_{X / S} \otimes  x^\ast F \ar[r] \ar[d, equal] & x^\ast \cJ^1_{X/ S}(F) \ar[r] \ar[d, "\phi"]&  x^\ast F \ar[r] \ar[d, equal] & 0 \\
0 \ar[r] &  x^\ast \Omega^1_{X / S} \otimes  x^\ast F \ar[r] &  \pi_\ast \iota^\ast F \ar[r] &  x^\ast F \ar[r] & 0
\end{tikzcd}
\]
of $\cO_S$-modules. The five lemma implies that $\phi$ is an isomorphism.
\end{proof}

%%%%%%%%%%%%%%%%%%%%%
% References
%%%%%%%%%%%%%%%%%%%%%

\end{document}